\documentclass[reqno,english]{amsart}
\usepackage[T1]{fontenc}
\usepackage[latin9]{inputenc}
\usepackage{babel}
\usepackage{array}
\usepackage{refstyle}
\usepackage{float}
\usepackage{units}
\usepackage{mathrsfs}
\usepackage{mathtools}
\usepackage{enumitem}
\usepackage{amstext}
\usepackage{amsthm}
\usepackage{amssymb}
\usepackage{graphicx}
\usepackage[all]{xy}
\PassOptionsToPackage{normalem}{ulem}
\usepackage{ulem}
\usepackage[unicode=true,pdfusetitle,
 bookmarks=true,bookmarksnumbered=false,bookmarksopen=false,
 breaklinks=false,pdfborder={0 0 1},backref=false,colorlinks=false]
 {hyperref}
\hypersetup{
 colorlinks=true,citecolor=blue,linkcolor=blue,linktocpage=true}

\makeatletter

\AtBeginDocument{\providecommand\thmref[1]{\ref{thm:#1}}}
\AtBeginDocument{\providecommand\secref[1]{\ref{sec:#1}}}
\AtBeginDocument{\providecommand\figref[1]{\ref{fig:#1}}}
\AtBeginDocument{\providecommand\defref[1]{\ref{def:#1}}}
\AtBeginDocument{\providecommand\lemref[1]{\ref{lem:#1}}}
\AtBeginDocument{\providecommand\corref[1]{\ref{cor:#1}}}
\AtBeginDocument{\providecommand\tabref[1]{\ref{tab:#1}}}
\AtBeginDocument{\providecommand\exaref[1]{\ref{exa:#1}}}
\AtBeginDocument{\providecommand\remref[1]{\ref{rem:#1}}}
\providecommand{\tabularnewline}{\\}
\RS@ifundefined{subsecref}
  {\newref{subsec}{name = \RSsectxt}}
  {}
\RS@ifundefined{thmref}
  {\def\RSthmtxt{theorem~}\newref{thm}{name = \RSthmtxt}}
  {}
\RS@ifundefined{lemref}
  {\def\RSlemtxt{lemma~}\newref{lem}{name = \RSlemtxt}}
  {}

\numberwithin{equation}{section}
\numberwithin{figure}{section}
\numberwithin{table}{section}
\theoremstyle{plain}
\newtheorem{thm}{\protect\theoremname}[section]
\theoremstyle{plain}
\newtheorem{lem}[thm]{\protect\lemmaname}
\theoremstyle{definition}
\newtheorem{defn}[thm]{\protect\definitionname}
\theoremstyle{remark}
\newtheorem{rem}[thm]{\protect\remarkname}
\theoremstyle{definition}
\newtheorem{example}[thm]{\protect\examplename}
\theoremstyle{plain}
\newtheorem{cor}[thm]{\protect\corollaryname}
\theoremstyle{remark}
\newtheorem*{acknowledgement*}{\protect\acknowledgementname}

\@ifundefined{date}{}{\date{}}
\allowdisplaybreaks
\usepackage{needspace}
\usepackage{refstyle}
\usepackage{enumitem}
\usepackage{bbm}

\newref{lem}{refcmd={Lemma \ref{#1}}}
\newref{thm}{refcmd={Theorem \ref{#1}}}
\newref{cor}{refcmd={Corollary \ref{#1}}}
\newref{sec}{refcmd={Section \ref{#1}}}
\newref{sub}{refcmd={Section \ref{#1}}}
\newref{subsec}{refcmd={Section \ref{#1}}}
\newref{chap}{refcmd={Chapter \ref{#1}}}
\newref{prop}{refcmd={Proposition \ref{#1}}}
\newref{exa}{refcmd={Example \ref{#1}}}
\newref{tab}{refcmd={Table \ref{#1}}}
\newref{rem}{refcmd={Remark \ref{#1}}}
\newref{def}{refcmd={Definition \ref{#1}}}
\newref{fig}{refcmd={Figure \ref{#1}}}

\setlist[enumerate]{itemsep=5pt,topsep=3pt}
\setlist[enumerate,1]{label=\textup{(}\roman*\textup{)},ref=\roman*}
\setlist[enumerate,2]{label=(\alph*),ref=\theenumi \alph*}

\usepackage{tikz}
\usetikzlibrary{decorations.pathreplacing}
\newcommand{\tikzmark}[1]{\tikz[overlay,remember picture,baseline=(#1.base)]
  \node (#1) {\strut};}

\AtBeginDocument{
  
}

\makeatother

\providecommand{\acknowledgementname}{Acknowledgement}
\providecommand{\corollaryname}{Corollary}
\providecommand{\definitionname}{Definition}
\providecommand{\examplename}{Example}
\providecommand{\lemmaname}{Lemma}
\providecommand{\remarkname}{Remark}
\providecommand{\theoremname}{Theorem}

\begin{document}
\title[]{Decomposition of Gaussian processes, and factorization of positive
definite kernels}
\author{Palle Jorgensen}
\address{(Palle E.T. Jorgensen) Department of Mathematics, The University of
Iowa, Iowa City, IA 52242-1419, U.S.A. }
\email{palle-jorgensen@uiowa.edu}
\urladdr{http://www.math.uiowa.edu/\textasciitilde jorgen/}
\author{Feng Tian}
\address{(Feng Tian) Department of Mathematics, Hampton University, Hampton,
VA 23668, U.S.A.}
\email{feng.tian@hamptonu.edu}
\begin{abstract}
We establish a duality for two factorization questions, one for general
positive definite (p.d) kernels $K$, and the other for Gaussian processes,
say $V$. The latter notion, for Gaussian processes is stated via
Ito-integration. Our approach to factorization for p.d. kernels is
intuitively motivated by matrix factorizations, but in infinite dimensions,
subtle measure theoretic issues must be addressed. Consider a given
p.d. kernel $K$, presented as a covariance kernel for a Gaussian
process $V$. We then give an explicit duality for these two seemingly
different notions of factorization, for p.d. kernel $K$, vs for Gaussian
process $V$. Our result is in the form of an explicit correspondence.
It states that the analytic data which determine the variety of factorizations
for $K$ is the exact same as that which yield factorizations for
$V$. Examples and applications are included: point-processes, sampling
schemes, constructive discretization, graph-Laplacians, and boundary-value
problems.
\end{abstract}


\maketitle
\tableofcontents{}

\section{Introduction}

We give an integrated approach to positive definite (p.d.) kernels
and Gaussian processes, with an emphasis on factorizations, and their
applications. Positive definite kernels serve as powerful tools in
such diverse areas as Fourier analysis, probability theory, stochastic
processes, boundary theory, potential theory, approximation theory,
interpolation, signal/image analysis, operator theory, spectral theory,
mathematical physics, representation theory, complex function-theory,
moment problems, integral equations, numerical analysis, boundary-value
problems for partial differential equations, machine learning, geometric
embedding problems, and information theory. While there is no single
book which covers all these applications, the reference \cite{MR3526117}
goes some of the way. As for the use of RKHS analysis in machine learning,
we refer to \cite{MR2327597} and \cite{MR3236858}.

Here, we give a new and explicit duality for positive definite functions
(kernels) on the one hand, and Gaussian processes on the other. A
covariance kernel for a general stochastic process is positive definite.
In general, the stochastic process in question is not determined by
its covariance kernel. But in the special case when the process is
Gaussian, it is. In fact (\thmref{C1}), every p.d. kernel $K$ is
indeed the covariance kernel of a Gaussian process. The construction
is natural; starting with the p.d. kernel $K$, there is a canonical
inductive limit construction leading to the Gaussian process for this
problem, following a realization of Gaussian processes dating back
to Kolmogorov. The interplay between analytic properties of p.d. kernels
and their associated Gaussian processes is the focus of our present
study.

We formulate two different factorization questions, one for general
p.d. kernels $K$, and the other for Gaussian processes, say $V$.
The latter notion, for Gaussian processes, is a subordination approach.
Our approach to factorization for p.d. kernels is directly motivated
by matrix factorizations, but in infinite dimensions, there are subtle
measure theoretic issues involved. If the given p.d. kernel $K$ is
already presented as a covariance kernel for a Gaussian process $V$,
we then give an explicit duality for these two seemingly different
notions of factorization. Our main result, \thmref{E1}, states that
the analytic data which determine the variety of factorizations for
$K$ is the exact same as that which yield factorizations for $V$.

\section{\label{sec:pdk}Positive definite kernels}

The notion of a positive definite (p.d.) kernel has come to serve
as a versatile tool in a host of problems in pure and applied mathematics.
The abstract notion of a p.d. kernel is in fact a generalization of
that of a positive definite function, or a positive-definite matrix.
Indeed, the matrix-point of view lends itself naturally to the particular
factorization question which we shall address in \secref{fac} below.
The general idea of p.d. kernels arose first in various special cases
in the first half of 20th century: It occurs in work by J. Mercer
in the context of solving integral operator equations; in the work
of G. Szeg\H{o} and S. Bergmann in the study of harmonic analysis
and the theory of complex domains; and in the work by N. Aronszajn
in boundary value problems for PDEs. It was Aronszajn who introduced
the natural notion of reproducing kernel Hilbert space (RKHS) which
will play a central role here; see especially (\ref{eq:A4}) below.
References covering the areas mentioned above include: \cite{MR3687240,MR0051437,MR562914,IM65,jorgensen2018harmonic,MR0277027,MR3882025},
and \cite{MR3721329}.

Right up to the present, p.d. kernels have arisen as powerful tools
in many and diverse areas of mathematics. A partial list includes
the areas listed above in the Introduction. An important new area
of application of RKHS theory includes the following \cite{MR1120274,MR1200633,MR1473250,MR1821907,MR1873434,MR1986785,MR2223568,MR2373103}.

\subsection*{Positive definite kernels and their reproducing kernel Hilbert spaces}

Let $X$ be a set and let $K$ be a complex valued function on $X\times X$.
We say that $K$ is \emph{positive definite} (p.d.) iff (Def.) for
all finite subset $F$ ($\subset X$) and complex numbers $\left(\xi_{x}\right)_{x\in F}$,
we have: 
\begin{equation}
\sum_{x\in F}\sum_{y\in F}\overline{\xi}_{x}\xi_{y}K\left(x,y\right)\geq0.\label{eq:A1}
\end{equation}
In other words, the $\left|F\right|\times\left|F\right|$ matrix $\left(K\left(x,y\right)\right)_{F\times F}$
is positive definite in the usual sense of linear algebra. We refer
to the rich literature regarding theory and applications of p.d. functions
\cite{MR2966130,MR3507188,HKL14,RAKK05,MR3290453,MR3046303,MR2982692}.

We shall also need the Aronszajn \cite{MR0051437} reproducing kernel
Hilbert spaces (R.K.H.S.), denoted $\mathscr{H}\left(K\right)$: It
is the Hilbert completion of all functions 
\begin{equation}
\sum_{x\in F}\xi_{x}K\left(\cdot,x\right)\label{eq:A2}
\end{equation}
 where $F$, and $\left(\xi\right)_{x\in F}$, are as above.

If $F$ (finite) is fixed, and $\left(\xi_{x}\right)_{x\in F}$, $\left(\eta_{x}\right)_{x\in F}$
are vectors in $\mathbb{C}^{\left|F\right|}$, we set 
\begin{equation}
\left\langle \sum\nolimits _{x\in F}\xi_{x}K\left(\cdot,x\right),\sum\nolimits _{y\in F}\eta_{y}K\left(\cdot,y\right)\right\rangle _{\mathscr{H}\left(K\right)}:=\sum\sum\nolimits _{F\times F}\overline{\xi}_{x}\eta_{y}K\left(x,y\right).\label{eq:A3}
\end{equation}

With the definition of the R.K.H.S. $\mathscr{H}\left(K\right)$,
we get directly that the functions $\left\{ K\left(\cdot,x\right)\right\} _{x\in X}$
are automatically in $\mathscr{H}\left(K\right)$; and that, for all
$h\in\mathscr{H}\left(K\right)$, we have 
\begin{equation}
\left\langle K\left(\cdot,x\right),h\right\rangle _{\mathscr{H}\left(K\right)}=h\left(x\right);\label{eq:A4}
\end{equation}
i.e., the \emph{reproducing} property holds.

Further recall (see e.g. \cite{MR3526117}) that, given $K$, then
the R.K.H.S. $\mathscr{H}\left(K\right)$ is determined uniquely,
up to isometric isomorphism in Hilbert space.
\begin{lem}
\label{lem:B1}Let $X\times X\xrightarrow{\;K\;}\mathbb{C}$ be a
p.d. kernel, and let $\mathscr{H}\left(K\right)$ be the corresponding
RKHS (see (\ref{eq:A3})-(\ref{eq:A4})). Let $h$ be a function defined
on $X$; then TFAE:
\begin{enumerate}
\item \label{enu:LemA2-1}$h\in\mathscr{H}\left(K\right)$;
\item \label{enu:LemA2-2}there is a constant $C=C_{h}<\infty$ such that,
for all finite subset $F\subset X$, and all $\left(\xi_{x}\right)_{x\in F}$,
$\xi_{x}\in\mathbb{C}$, the following \uline{a priori} estimate
holds:
\begin{equation}
\left|\sum\nolimits _{x\in F}\xi_{x}h\left(x\right)\right|^{2}\leq C_{h}\sum\nolimits _{x\in F}\sum\nolimits _{y\in F}\overline{\xi}_{x}\xi_{y}K\left(x,y\right).\label{eq:B5}
\end{equation}
\end{enumerate}
\end{lem}

\begin{proof}
The implication (\ref{enu:LemA2-1})$\Rightarrow$(\ref{enu:LemA2-2})
is immediate, and in this case, we may take $C_{h}=\left\Vert h\right\Vert _{\mathscr{H}\left(K\right)}^{2}$.

Now for the converse, assume (\ref{enu:LemA2-2}) holds for some finite
constant. On the $\mathscr{H}\left(K\right)$-dense span in (\ref{eq:A2}),
define a linear functional 
\begin{equation}
L_{h}\left(\sum\nolimits _{x\in F}\xi_{x}K\left(\cdot,x\right)\right):=\sum\nolimits _{x\in F}\xi_{x}h\left(x\right).\label{eq:B6}
\end{equation}
From the assumption (\ref{eq:B5}) in (\ref{enu:LemA2-2}), we conclude
that $L_{h}$ (in (\ref{eq:B6})) is a well defined bounded linear
functional on $\mathscr{H}\left(K\right)$. Initially, $L_{h}$ is
only defined on the span (\ref{eq:A2}), but by (\ref{eq:B5}), it
is bounded, and so extends uniquely by $\mathscr{H}\left(K\right)$-norm
limits. We may therefore apply Riesz' lemma to the Hilbert space $\mathscr{H}\left(K\right)$,
and conclude that there is a unique $H\in\mathscr{H}\left(K\right)$
such that 
\begin{equation}
L_{h}\left(\psi\right)=\left\langle \psi,H\right\rangle _{\mathscr{H}\left(K\right)}\label{eq:B7}
\end{equation}
for all $\psi\in\mathscr{H}\left(K\right)$. Now, setting $\psi\left(\cdot\right):=K\left(\cdot,x\right)$,
for $x\in X$, we conclude from (\ref{eq:B7}) that $h\left(x\right)=H\left(x\right)$;
and so $h\in\mathscr{H}\left(K\right)$, proving (\ref{enu:LemA2-1}).
\end{proof}

\section{\label{sec:gp}Gaussian processes}

The interest in positive definite (p.d.) functions has at least three
roots: (i) Fourier analysis, and harmonic analysis more generally;
(ii) Optimization and approximation problems, involving for example
spline approximations as envisioned by I. Sch\"oenberg; and (iii)
Stochastic processes. See \cite{MR0004644,MR709376}.

Below, we sketch a few details regarding (iii). A \emph{stochastic
process} is an indexed family of random variables based on a fixed
probability space. In some cases, the processes will be indexed by
some group $G$, or by a subset of $G$. For example, $G=\mathbb{R}$,
or $G=\mathbb{Z}$, correspond to processes indexed by real time,
respectively discrete time. A main tool in the analysis of stochastic
processes is an associated \emph{covariance function}.

A process $\{X_{g}\mid g\in G\}$ is called \emph{Gaussian} if each
random variable $X_{g}$ is Gaussian, i.e., its distribution is Gaussian.
For Gaussian processes, we only need two moments. So if we normalize,
setting the mean equal to $0$, then the process is determined by
its covariance function. In general, the covariance function is a
function on $G\times G$, or on a subset, but if the process is \emph{stationary},
the covariance function will in fact be a p.d. function defined on
$G$, or a subset of $G$. For a systematic study of positive definite
functions on groups $G$, on subsets of groups, and the variety of
the extensions to p.d. functions on $G$, see e.g. \cite{MR3559001}.

By a theorem of Kolmogorov \cite{MR735967}, every Hilbert space may
be realized as a (Gaussian) reproducing kernel Hilbert space (RKHS),
see \thmref{C1} below, and also \cite{PaSc75,IM65,NF10}.

Now every positive definite kernel is also the covariance kernel of
a Gaussian process; a fact which is a point of departure in our present
analysis: Given a positive definite kernel, we shall explore its use
in the analysis of the associated Gaussian process; and vice versa.

This point of view is especially fruitful when one is dealing with
problems from stochastic analysis. Even restricting to stochastic
analysis, we have the exciting area of applications to statistical
learning theory \cite{MR2327597,MR3236858}.\\

Let $\left(\Omega,\mathscr{F},\mathbb{P}\right)$ be a \emph{probability
space}, i.e., $\Omega$ is a fixed set (sample space), $\mathscr{F}$
is a specified sigma-algebra (events) of subsets in $\Omega$, and
$\mathbb{P}$ is a probability measure on $\mathscr{F}$.

A Gaussian random variable is a function $V:\Omega\rightarrow\mathbb{R}$
(in the real case), or $V:\Omega\rightarrow\mathbb{C}$, such that
$V$ is measurable with respect to the sigma-algebra $\mathscr{F}$
on $\Omega$, and the corresponding sigma-algebra of Borel subsets
in $\mathbb{R}$ (or in $\mathbb{C}$). Let $\mathbb{E}$ denote the
expectation defined from $\mathbb{P}$, i.e., 
\begin{equation}
\mathbb{E}\left(\cdots\right)=\int_{\Omega}\left(\cdots\right)d\mathbb{P}.\label{eq:A5}
\end{equation}
The requirement on $V$ is that its distribution is Gaussian. If $g$
denotes a Gaussian on $\mathbb{R}$ (or on $\mathbb{C}$), the requirement
is that 
\begin{equation}
\mathbb{E}\left(f\circ V\right)=\int_{\mathbb{R}\left(\text{or \ensuremath{\mathbb{C}}}\right)}f\,dg;\label{eq:A6}
\end{equation}
or equivalently
\begin{equation}
\mathbb{P}\left(V\in B\right)=\int_{B}dg=g\left(B\right)\label{eq:A7}
\end{equation}
for all Borel sets $B$; see \figref{G1}.

\begin{figure}
\includegraphics[width=0.35\textwidth]{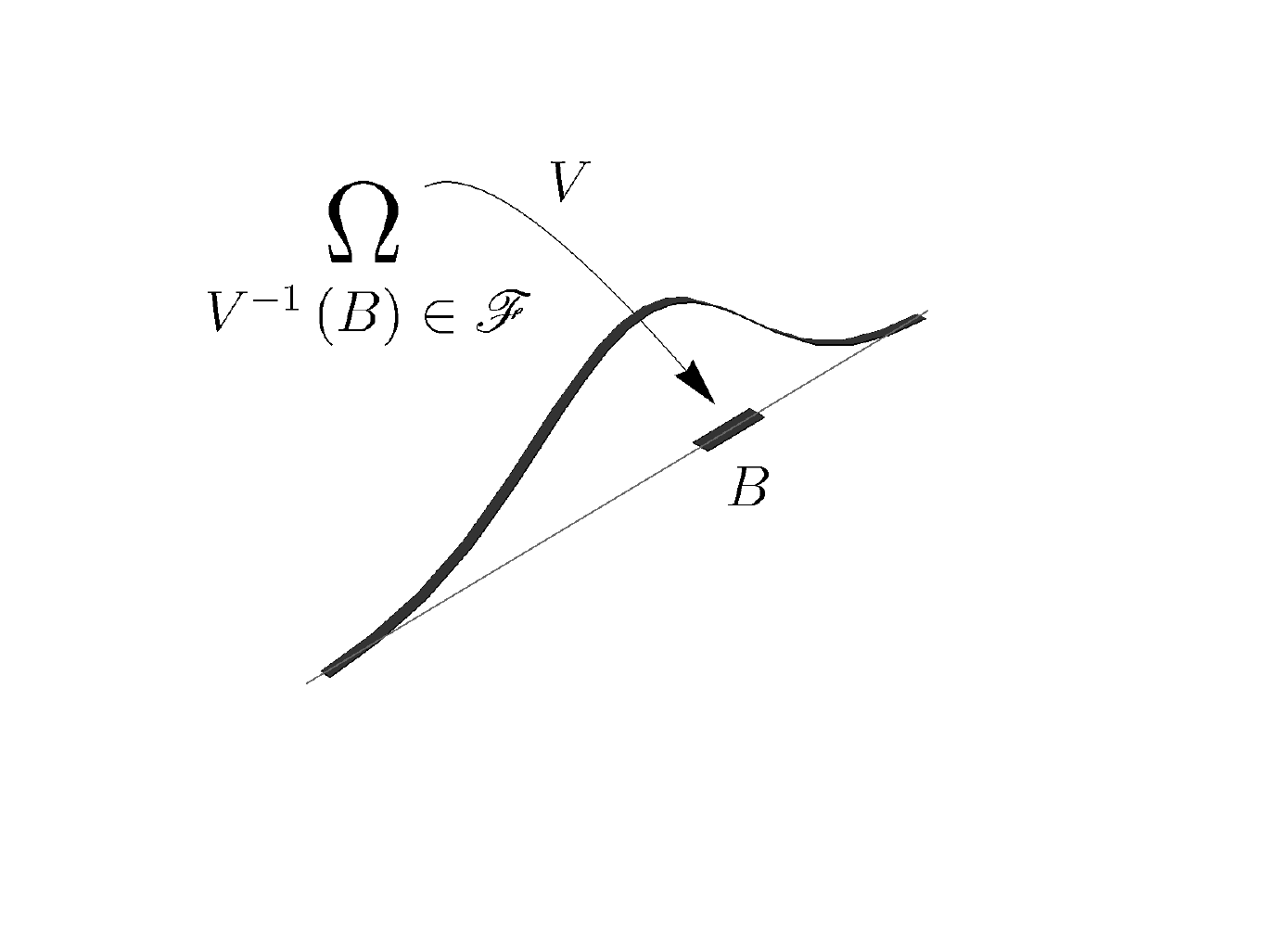}

\caption{\label{fig:G1}A Gaussian random variable and its distribution, see
(\ref{eq:A7}).}
\end{figure}

If $N\in\mathbb{N}$, and $V_{1},\cdots,V_{N}$ are random variables,
the Gaussian requirement is (see \figref{G2}) that the joint distribution
of $\left(V_{1},\cdots,V_{N}\right)$ is an $N$-dimensional Gaussian,
say $g_{N}$, so if $B\subset\mathbb{R}^{N}$ then 
\begin{equation}
\mathbb{P}\left(\left(V_{1},\cdots,V_{N}\right)\in B\right)=g_{N}\left(B\right).\label{eq:A8}
\end{equation}

\begin{figure}
\includegraphics[width=0.35\textwidth]{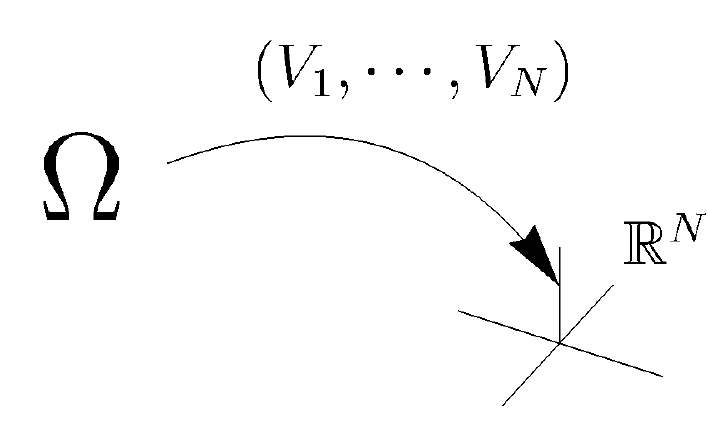}

\caption{\label{fig:G2} A Gaussian system and its joint distribution, see
(\ref{eq:A8}).}

\end{figure}

For our present purpose we may restrict to the case where the mean
(of the respective Gaussians) is assumed zero. In that case, a finite
joint distribution is determined by its covariance matrix. In the
$\mathbb{R}^{N}$ case, it is specified as follows (the extension
to $\mathbb{C}^{N}$ is immediate) $\left(G_{N}\left(j_{1},j_{2}\right)\right)_{j_{1},j_{2}=1}^{N}$,
\begin{equation}
G_{N}\left(j_{1},j_{2}\right)=\int_{\mathbb{R}^{N}}x_{j_{1}}x_{j_{2}}g_{N}\left(x_{1},\cdots,x_{N}\right)dx_{1}\cdots dx_{N}\label{eq:A9}
\end{equation}
where $dx_{1}\cdots dx_{N}=\lambda_{N}$ denotes the standard Lebesgue
measure on $\mathbb{R}^{N}$.

The following is known:
\begin{thm}[Kolmogorov \cite{MR0133175}, see also \cite{MR562914,MR1176778}]
\label{thm:C1}A kernel $K:X\times X\rightarrow\mathbb{C}$ is positive
definite if and only if there is a (mean zero) Gaussian process $\left(V_{x}\right)_{x\in X}$
indexed by $X$ such that 
\begin{equation}
\mathbb{E}\left(\overline{V}_{x}V_{y}\right)=K\left(x,y\right)\label{eq:A10}
\end{equation}
where $\overline{V}_{x}$ denotes complex conjugation.

Moreover (see Hida \cite{MR0301806,MR1176778}), the process in (\ref{eq:A10})
is uniquely determined by the kernel $K$ in question. If $F\subset X$
is finite, then the covariance kernel for $\left(V_{x}\right)_{x\in F}$
is $K_{F}$ given by 
\begin{equation}
K_{F}\left(x,y\right)=G_{F}\left(x,y\right),\label{eq:A11}
\end{equation}
for all $x,y\in F$, see (\ref{eq:A9}) above.
\end{thm}

In the subsequent sections, we shall address a number of properties
of Gaussian processes important for their stochastic calculus. Our
analysis deals with both the general case, and particular examples
from applications. We begin in \secref{sms} with certain Wiener processes
which are indexed by sigma-finite measures. For this class, the corresponding
p.d. kernel has a special form; see (\ref{eq:C1}) in \defref{D1}.
(The case of fractal measures is part of \secref{exa} below.) In
\secref{fac}, we address the general case: We prove our duality result
for factorization, \thmref{E1}. The remaining sections are devoted
to examples and applications.

\section{\label{sec:sms}Sigma-finite measure spaces and Gaussian processes}

We shall consider functions of $\sigma$-finite measure space $\left(M,\mathscr{F}_{M},\mu\right)$
where $M$ is a set, $\mathscr{F}_{M}$ a $\sigma$-algebra of subsets
in $M$, and $\mu$ is a positive measure defined on $\mathscr{F}_{M}$.
It is further assumed that there is a countably indexed $\left(A_{i}\right)_{i\in\mathbb{N}}$
s.t. $0<\mu\left(A_{i}\right)<\infty$, $M=\cup_{i}A_{i}$; and further
that the measure space $\left(M,\mathscr{F}_{M},\mu\right)$ is complete;
so the Radon-Nikodym theorem holds. We shall also restrict to the
case when $\mu$ is assumed non-atomic. The case when $\mu$ is atomic
is different, and is addressed in \secref{atomic} below.
\begin{defn}
\label{def:D1}Set 
\[
\mathscr{F}_{fin}=\left\{ A\in\mathscr{F}_{M}\mid0<\mu\left(A\right)<\infty\right\} .
\]
\end{defn}

Note then 
\begin{equation}
K^{\left(\mu\right)}\left(A,B\right)=\mu\left(A\cap B\right),\;A,B\in\mathscr{F}_{fin}\label{eq:C1}
\end{equation}
is positive definite. The corresponding Gaussian process $(W_{A}^{\left(\mu\right)})_{A\in\mathscr{F}_{fin}}$
is called the Wiener process \cite{MR0301806,MR1176778}. In particular,
we have 
\begin{equation}
\mathbb{E}\left(W_{A}^{\left(\mu\right)}W_{B}^{\left(\mu\right)}\right)=\mu\left(A\cap B\right),\label{eq:C2}
\end{equation}
and 
\begin{equation}
\lim_{\left(A_{i}\right)}\sum_{i}\left(W_{A_{i}}^{\left(\mu\right)}\right)^{2}=\mu\left(A\right).\label{eq:C3}
\end{equation}
The precise limit in (\ref{eq:C3}), quadratic variation, is as follows:
Given $\mu$ as above, and $A\in\mathscr{F}_{fin}$, we then take
limit over the filter of all partitions of $A$ (see (\ref{eq:C4}))
relative to the standard notation of refinement:
\begin{equation}
A=\cup_{i}A_{i},\;A_{i}\cap A_{j}=\emptyset\;\text{if \ensuremath{i\neq j}},\;\text{and}\;\lim\mu\left(A_{i}\right)=0.\label{eq:C4}
\end{equation}

Details: Let $\left(\Omega,Cyl,\mathbb{P}\right)$, $\mathbb{P}=\mathbb{P}^{\left(\mu\right)}$
be the probability space which realizes $W^{\left(\mu\right)}$ as
a Gaussian process (or generalized Wiener process), i.e., s.t. (\ref{eq:C2})
holds for all pairs in $\mathscr{F}_{fin}$. In particular, we have
that $W_{A}^{\left(\mu\right)}\underset{\left(\text{dist}\right)}{\sim}N\left(0,\mu\left(A\right)\right)$,
i.e., mean zero, Gaussian, and variance = $\mu\left(A\right)$. Then:
\begin{lem}[see e.g., \cite{MR3687240}]
With the assumptions as above, we have 
\begin{equation}
\lim_{\left(A_{i}\right)}\mathbb{E}\left(\big|\mu\left(A\right)\mathbbm{1}-\sum\nolimits _{i}(W_{A_{i}}^{\left(\mu\right)})^{2}\Big|^{2}\right)=0\label{eq:C4-1}
\end{equation}
where (in (\ref{eq:C4-1})) the limit is taken over the filter of
all partitions $\left(A_{i}\right)$ of $A$, and $\mathbbm{1}$ denotes
the constant function ``one'' on $\Omega$.
\end{lem}

As a result, we get the following Ito-integral 
\begin{equation}
W^{\left(\mu\right)}\left(f\right):=\int_{M}f\left(s\right)\,dW_{s}^{\left(\mu\right)},\label{eq:C6}
\end{equation}
defined for all $f\in L^{2}\left(M,\mathscr{F},\mu\right)$, and 
\begin{equation}
\mathbb{E}\left(\left|\int\nolimits _{M}f\left(s\right)dW_{s}^{\left(\mu\right)}\right|^{2}\right)=\int_{M}\left|f\left(s\right)\right|^{2}d\mu\left(s\right).\label{eq:C5}
\end{equation}
We note that the following operator, 
\begin{equation}
L^{2}\left(M,\mu\right)\ni f\longmapsto W^{\left(\mu\right)}\left(f\right)\in L^{2}\left(\Omega,\mathbb{P}\right)\label{eq:C7}
\end{equation}
is isometric.

In our subsequent considerations, we shall need the following precise
formula (see \lemref{D3}) for the RKHS associated with the p.d. kernel
\begin{equation}
K^{\left(\mu\right)}\left(A,B\right):=\mu\left(A\cap B\right),\label{eq:C8}
\end{equation}
defined on $\mathscr{F}_{fin}\times\mathscr{F}_{fin}$. We denote
the RKHS by $\mathscr{H}(K^{\left(\mu\right)})$.
\begin{lem}
\label{lem:D3}Let $\mu$ be as above, and let $K^{\left(\mu\right)}$
be the p.d. kernel on $\mathscr{F}_{fin}$ defined in (\ref{eq:C8}).
Then the corresponding RKHS $\mathscr{H}(K^{\left(\mu\right)})$ is
as follows: A function $\Phi$ on $\mathscr{F}_{fin}$ is in $\mathscr{H}(K^{\left(\mu\right)})$
if and only if there is a $\varphi\in L^{2}\left(M,\mathscr{F}_{M},\mu\right)\left(=:L^{2}\left(\mu\right)\right)$
such that 
\begin{equation}
\Phi\left(A\right)=\int_{A}\varphi\,d\mu,\label{eq:C10}
\end{equation}
for all $A\in\mathscr{F}_{fin}$. Then 
\begin{equation}
\left\Vert \Phi\right\Vert _{\mathscr{H}(K^{\left(\mu\right)})}=\left\Vert \varphi\right\Vert _{L^{2}\left(\mu\right)}.\label{eq:D11-1}
\end{equation}
\end{lem}

\begin{proof}
To show that $\Phi$ in (\ref{eq:C10}) is in $\mathscr{H}(K^{\left(\mu\right)})$,
we must choose a finite constant $C_{\Phi}$ such that, for all finite
subset $\left(A_{i}\right)_{i=1}^{N}$, $A_{i}\in\mathscr{F}_{fin}$,
$\left\{ \xi_{i}\right\} _{i=1}^{N}$, $\xi_{i}\in\mathbb{R}$, we
get the following \emph{a priori} estimate:
\begin{equation}
\left|\sum\nolimits _{i=1}^{N}\xi_{i}\Phi\left(A_{i}\right)\right|^{2}\leq C_{\Phi}\sum\nolimits _{i}\sum\nolimits _{j}\xi_{i}\xi_{j}K^{\left(\mu\right)}\left(A_{i},A_{j}\right).\label{eq:C11}
\end{equation}
But a direct application of Schwarz to $L^{2}\left(\mu\right)$ shows
that (\ref{eq:C11}) holds, and for a finite $C_{\Phi}$, we may take
$C_{\Phi}=\left\Vert \varphi\right\Vert _{L^{2}\left(\mu\right)}^{2}$,
where $\varphi$ is the $L^{2}\left(\mu\right)$-function in (\ref{eq:C10}).
The desired conclusion now follows from an application of \lemref{B1}.

We have proved one implication from the statement of the lemma: Functions
$\Phi$ on $\mathscr{F}_{fin}$ of the formula (\ref{eq:C10}) are
in the RKHS $\mathscr{H}\left(K^{\left(\mu\right)}\right)$, and the
norm $\left\Vert \cdot\right\Vert _{\mathscr{H}\left(K^{\left(\mu\right)}\right)}$
is as stated in (\ref{eq:D11-1}). In the below, we shall denote these
elements in $\mathscr{H}\left(K^{\left(\mu\right)}\right)$ as pairs
$\left(\Phi,\varphi\right)$. We shall also restrict attention to
the case of real valued functions.

For the converse implication, let $H$ be a function on $\mathscr{F}_{fin}$,
and assume $H\in\mathscr{H}\left(K^{\left(\mu\right)}\right)$. Then
by Schwarz applied to $\left\langle \cdot,\cdot\right\rangle _{\mathscr{H}\left(K^{\left(\mu\right)}\right)}$
we get 
\begin{equation}
\left|\left\langle H,\Phi\right\rangle _{\mathscr{H}\left(K^{\left(\mu\right)}\right)}\right|\leq\left\Vert H\right\Vert _{\mathscr{H}\left(K^{\left(\mu\right)}\right)}\left\Vert \varphi\right\Vert _{L^{2}\left(\mu\right)},\label{eq:D11-2}
\end{equation}
where we used (\ref{eq:D11-1}). Hence when Schwarz is applied to
$L^{2}\left(\mu\right)$, we get a unique $h\in L^{2}\left(\mu\right)$
such that 
\begin{equation}
\left\langle H,\Phi\right\rangle _{\mathscr{H}\left(K^{\left(\mu\right)}\right)}=\int_{M}h\,\varphi\,d\mu\label{eq:D11-3}
\end{equation}
for all $\left(\Phi,\varphi\right)$ as in (\ref{eq:C10}). Now specialize
to $\varphi=\chi_{A}$, $A\in\mathscr{F}_{fin}$, in (\ref{eq:D11-3})
and we conclude that 
\begin{equation}
H\left(A\right)=\int_{A}h\,d\mu;
\end{equation}
which translates into the assertion that the pair $\left(H,h\right)$
has the desired form (\ref{eq:C10}). And hence by (\ref{eq:D11-1})
we have $\left\Vert H\right\Vert _{\mathscr{H}\left(K^{\left(\mu\right)}\right)}=\left\Vert h\right\Vert _{L^{2}\left(\mu\right)}$
as stated. This concludes the proof of the converse inclusion.
\end{proof}

\section{\label{sec:fac}Factorizations and stochastic integrals}

In Sections \ref{sec:pdk} and \ref{sec:gp}, we introduced the related
notions of positive definite (p.d.) functions (kernels) on the one
hand, and Gaussian processes on the other. One notes the immediate
fact that a covariance kernel for a general stochastic process is
positive definite. In general, the stochastic process in question
is not determined by its covariance kernel. But in the special case
when the process is Gaussian, it is.

In \thmref{C1}, we stated that every p.d. kernel $K$ is indeed the
covariance kernel of a Gaussian process. The construction is natural;
starting with the p.d. kernel $K$, there is a canonical inductive
limit construction leading to the Gaussian process for this problem.
The basic idea for this particular construction of Gaussian processes
dates back to pioneering work by Kolmogorov \cite{MR735967,MR562914}.

In the present section, we formulate two different factorization questions,
one for general p.d. kernels $K$, and the other for Gaussian processes,
say $V$. For details, see the respective definitions in (\ref{eq:D2})
and (\ref{eq:D3}) below. If $K$ is indeed the covariance kernel
for a Gaussian process $V$, it is natural to try to relate these
two seemingly different notions of factorization. (In the case of
Gaussian processes, a better name is perhaps \textquotedblleft subordination\textquotedblright{}
(see (\ref{eq:D6}) below), but our theorem justifies the use of factorization
in both of these contexts.) Our main result, \thmref{E1}, states
that the data determining factorization for $K$ is the exact same
as that which yields factorization for $V$.\\

Let $K$ be a positive definite kernel $X\times X\xrightarrow{\;K\;}\mathbb{C}$;
and let $V=V_{K}$ be the corresponding Gaussian (mean zero) process,
indexed by $X$, i.e., $V_{x}\in L^{2}\left(\Omega,\mathbb{P}\right)$,
$\forall x\in X$, and 
\begin{equation}
\mathbb{E}\left(\overline{V}_{x}V_{y}\right)=K\left(x,y\right),\;\forall\left(x,y\right)\in X\times X.\label{eq:D1}
\end{equation}
We set 
\begin{align}
\mathscr{F}\left(K\right) & :=\Big\{\left(M,\mathscr{F}_{M},\mu\right)\mid\text{s.t. }K\left(\cdot,x\right)\longmapsto k_{x}\in L^{2}\left(M,\mu\right)\label{eq:D2}\\
 & \qquad\text{extends to an isometry, i.e., }\nonumber \\
 & \qquad K\left(x,y\right)=\int_{M}\overline{k_{x}\left(s\right)}k_{y}\left(s\right)d\mu\left(s\right)=\left\langle k_{x},k_{y}\right\rangle _{L^{2}\left(\mu\right)},\;\forall x,y\in X\Big\}.\nonumber 
\end{align}
Further, if $V$ is the Gaussian process (from (\ref{eq:D1})), we
set 
\begin{align}
\mathscr{M}\left(V\right) & :=\Big\{\left(M,\mathscr{F}_{M},\mu\right)\mid\text{s.t. \ensuremath{V} admits an Ito-integral representation}\nonumber \\
 & \qquad V_{x}=\int_{M}k_{x}\left(s\right)dW_{s}^{\left(\mu\right)},\;\forall x\in X,\text{ where}\:\left\{ k_{x}\right\} _{x\in X}\text{ is an}\label{eq:D3}\\
 & \qquad\text{indexed system in \ensuremath{L^{2}\left(M,\mu\right)}}\Big\}.\nonumber 
\end{align}

Following parallel terminology from measure theory, we say that a
Gaussian process $V$ admits a disintegration, via suitable Ito-integrals,
when there is a measure space with measure $\mu$ such that the corresponding
Wiener process $W^{\left(\mu\right)}$ satisfies (\ref{eq:D3}). Our
theorem below (\thmref{E1}) shows that this disintegration question
may be decided instead by the answer to an equivalent spectral decomposition
question; the latter of course formulated for the covariance kernel
for $V$. As is shown in the examples/applications below, given a
Gaussian process, it is not at all clear what disintegrations hold;
see for example \corref{F7}.
\begin{thm}
\label{thm:E1}Let $K:X\times X\rightarrow\mathbb{C}$ be given positive
definite, and let $\left\{ V_{x}\right\} _{x\in X}$ be the corresponding
Gaussian (mean zero) process, then 
\begin{equation}
\mathscr{F}\left(K\right)=\mathscr{M}\left(V\right).\label{eq:D4}
\end{equation}
\end{thm}

\begin{proof}
We shall need the following:
\end{proof}
\begin{lem}
\label{lem:F2}From the definition of $\mathscr{F}\left(K\right)$,
with $K$ fixed and assumed p.d., we get to every $\left(\left(k_{x}\right)_{x\in X},\mu\right)\in\mathscr{F}\left(K\right)$
a natural isometry $T_{\mu}:\mathscr{H}\left(K\right)\longrightarrow L^{2}\left(M,\mu\right)$.
It is denoted by
\begin{equation}
T_{\mu}(\underset{\in\mathscr{H}\left(K\right)}{\underbrace{K\left(\cdot,x\right)}}):=k_{x}\in L^{2}\left(\mu\right);\label{eq:D4-1}
\end{equation}
and the adjoint operator $T_{\mu}^{*}:L^{2}\left(M,\mu\right)\longrightarrow\mathscr{H}\left(K\right)$
is as follows: For all $f\in L^{2}\left(M,\mu\right)$ we have 
\begin{equation}
\left(T_{\mu}^{*}f\right)\left(x\right)=\int_{M}f\left(s\right)\overline{k_{x}\left(s\right)}d\mu\left(s\right).\label{eq:D4-2}
\end{equation}
Moreover, we also have 
\begin{equation}
T_{\mu}^{*}\left(k_{x}\right)=K\left(\cdot,x\right),\;\text{for all \ensuremath{x\in X}.}\label{eq:E7-1}
\end{equation}
\end{lem}

\begin{proof}
Since $\left(k_{x},\mu\right)\in\mathscr{F}\left(K\right)$, we have
the factorization property (\ref{eq:D2}), and so it follows from
(\ref{eq:D4-1}) that this extends by linearity and norm-completion
to an isometry $\mathscr{H}\left(K\right)\xrightarrow{\;T_{\mu}\;}L^{2}\left(\mu\right)$
as stated.

By the definition of the adjoint operator $L^{2}\left(\mu\right)\xrightarrow{\;T_{\mu}^{*}\;}\mathscr{H}\left(K\right)$,
we have for $f\in L^{2}\left(\mu\right)$: 
\[
\left(T_{\mu}^{*}f\right)\left(x\right)=\left\langle K\left(\cdot,x\right),T_{\mu}^{*}f\right\rangle _{\mathscr{H}\left(K\right)}=\left\langle k_{x},f\right\rangle _{L^{2}\left(\mu\right)}=\int_{M}f\left(s\right)\overline{k_{x}\left(s\right)}d\mu\left(s\right),
\]
which is the assertion in the lemma.

From the properties of $\mathscr{H}\left(K\right)$ (see \secref{pdk}),
it follows that (\ref{eq:E7-1}) holds iff 
\begin{equation}
\left\langle K\left(\cdot,y\right),T_{\mu}^{*}\left(k_{x}\right)\right\rangle _{\mathscr{H}\left(K\right)}=\left\langle K\left(\cdot,y\right),K\left(\cdot,x\right)\right\rangle _{\mathscr{H}\left(K\right)}\label{eq:E7-2}
\end{equation}
for all $y\in X$. But we may compute both sides in eq. (\ref{eq:E7-2})
as follows: 
\begin{eqnarray*}
\text{LHS}_{\left(\ref{eq:E7-2}\right)} & = & \left\langle T_{\mu}K\left(\cdot,y\right),k_{x}\right\rangle _{L^{2}\left(\mu\right)}\\
 & \underset{\text{by \ensuremath{\left(\ref{eq:D4-1}\right)}}}{=} & \left\langle k_{y},k_{x}\right\rangle _{L^{2}\left(\mu\right)}\\
 & \underset{\text{by \ensuremath{\left(\ref{eq:D2}\right)}}}{=} & K\left(y,x\right)\\
 & \underset{\text{by \ensuremath{\left(\ref{eq:A3}\right)}}}{=} & \text{RHS}_{\left(\ref{eq:E7-2}\right)}.
\end{eqnarray*}
\end{proof}
\begin{proof}[Proof of \thmref{E1} continued]
The proof is divided into two parts, one for each of the inclusions
$\subseteq$ and $\supseteq$ in (\ref{eq:D4}).

\textbf{Part 1 ``$\subseteq$''.} Assume a pair $\left(\left(k_{x}\right)_{x\in X},\mu\right)$
is in $\mathscr{F}\left(K\right)$; see (\ref{eq:D2}). Then by definition,
the factorization (\ref{eq:D3}) holds on $X\times X$. Now let $W^{\left(\mu\right)}$
denote the Wiener process associated with $\mu$, i.e., $W^{\left(\mu\right)}$
is a Gaussian process indexed by $\mathscr{F}_{fin}$, and 
\begin{equation}
\mathbb{E}\left(W_{A}^{\left(\mu\right)}W_{B}^{\left(\mu\right)}\right)=\mu\left(A\cap B\right),\label{eq:D5}
\end{equation}
for all $A,B\in\mathscr{F}_{fin}$; see (\ref{eq:C1}) above. Now
form the Ito-integral 
\begin{equation}
V_{x}:=\int_{M}k_{x}\left(s\right)dW_{s}^{\left(\mu\right)},\;x\in X.\label{eq:D6}
\end{equation}
We stress that then $V_{x}$, as defined by (\ref{eq:D6}), is a Gaussian
process indexed by $X$. To see this, use the general theory of Ito-integration,
see also \cite{MR3882025,MR3701541,MR3574374,MR3721329,MR3670916,MR0301806,MR562914}.
The approximation in (\ref{eq:D6}) is over the filter of all \emph{partitions
}
\begin{equation}
\left\{ A_{i}\right\} _{i\in\mathbb{N}}\text{ s.t. }A_{i}\cap A_{j}=\emptyset,\:i\neq j,\;\cup_{i\in\mathbb{N}}A_{i}=M,\;\text{and}\;0<\mu\left(A_{i}\right)<\infty;\label{eq:D7}
\end{equation}
see (\ref{eq:C4}). From the property of $W_{A_{i}}^{\left(\mu\right)}$,
$i\in\mathbb{N}$, we conclude that, for all $s_{i}\in A_{i}$, we
have that 
\begin{equation}
\sum_{i\in\mathbb{N}}k_{x}\left(s_{i}\right)W_{A_{i}}^{\left(\mu\right)}\label{eq:D8}
\end{equation}
is Gaussian (mean zero) with 
\begin{align}
\mathbb{E}\left|\sum\nolimits _{i}k_{x}\left(s_{i}\right)W_{A_{i}}^{\left(\mu\right)}\right|^{2} & =\sum\nolimits _{i}\sum\nolimits _{j}\overline{k_{x}\left(s_{i}\right)}k_{x}\left(s_{j}\right)\mu\left(A_{i}\cap A_{j}\right)\nonumber \\
 & =\sum\nolimits _{i}\left|k_{x}\left(s_{i}\right)\right|^{2}\mu\left(A_{i}\right);\label{eq:D9}
\end{align}
where we used (\ref{eq:D7}). Passing to the limit over the filter
of all partitions of $M$ (as in (\ref{eq:D7})), we then get 
\[
\mathbb{E}\left(\int_{M}\overline{k_{x}\left(s\right)}dW_{s}^{\left(\mu\right)}\int_{M}k_{y}\left(t\right)dW_{t}^{\left(\mu\right)}\right)=\int_{M}\overline{k_{x}\left(s\right)}k_{y}\left(s\right)d\mu\left(s\right);
\]
and with definition (\ref{eq:D6}), therefore: 
\begin{equation}
\mathbb{E}\left(\overline{V}_{x}V_{y}\right)=\left\langle k_{x},k_{y}\right\rangle _{L^{2}\left(\mu\right)}=K\left(x,y\right),\;\forall\left(x,y\right)\in X\times X,\label{eq:D10}
\end{equation}
where the last step in the derivation (\ref{eq:D10}) uses the assumption
that $\left(\left(k_{x}\right)_{x\in X},\mu\right)\in\mathscr{F}\left(K\right)$;
see (\ref{eq:D2}).

\textbf{Part 2 ``$\supseteq$''.} Assume now that some pair $\left(\left(k_{x}\right)_{x\in X},\mu\right)$
is in $\mathscr{M}\left(V\right)$ where $K$ is given assumed p.d.;
and where $\left(V_{x}\right)_{x\in X}$ is ``the'' associated (mean
zero) Gaussian process; i.e., with $K$ as its covariance kernel;
see (\ref{eq:D1}).

We claim that $\left(\left(k_{x}\right)_{x\in X},\mu\right)$ must
then be in $\mathscr{F}\left(K\right)$, i.e., that the factorization
(\ref{eq:D3}) holds. This in turn follows from the following chain
of identities: 
\begin{alignat}{2}
\left\langle k_{x},k_{y}\right\rangle _{L^{2}\left(\mu\right)} & =\mathbb{E}\left(\overline{V}_{x}V_{y}\right) & \quad & \big(\text{since \ensuremath{V_{x}=\int_{M}k_{x}\left(s\right)dW_{s}^{\left(\mu\right)}}}\big)\nonumber \\
 & =K\left(x,y\right) &  & \big(\text{since \ensuremath{K} is the covariance kernel}\label{eq:D11}\\
 &  &  & \text{ of the Gaussian process \ensuremath{\left(V_{x}\right)_{x\in X}}}\big)\nonumber 
\end{alignat}
valid for $\forall\left(x,y\right)\in X\times X$, and the conclusion
follows. Note that the first step in the derivation of (\ref{eq:D11})
uses the Ito-isometry. Hence, initially $K$ may possibly be the covariance
kernel for a mean zero Gaussian process, say $\left(V'_{x}\right)$,
different from $V_{x}:=\int_{M}k_{x}\left(s\right)dW_{s}^{\left(\mu\right)}$.
But we proved that the two Gaussian processes $V_{x}$, and $V_{x}'$,
have the same covariance kernel. It follows then the two processes
must be equivalent. This is by general theory; see e.g. \cite{MR0277027,MR2053326,MR3687240}.

The last uniqueness is only valid since we can consider Gaussian processes.
Other stochastic processes are typically not determined uniquely from
the respective covariance kernels.
\end{proof}
\begin{rem}
In the statement of \thmref{E1} there are two isometries: Starting
with $\left(\left(k_{x}\right)_{x\in X},\mu\right)\in\mathscr{F}\left(K\right)$
we get the canonical isometry $T_{\mu}:\mathscr{H}\left(K\right)\rightarrow L^{2}\left(\mu\right)$
given by 
\begin{equation}
T_{\mu}\left(K\left(\cdot,x\right)\right)=k_{x};
\end{equation}
see (\ref{eq:D4-1}) of \lemref{F2}. But with $\mu$, we then also
get the Wiener process $W^{\left(\mu\right)}$ and the Ito-integral
\begin{equation}
L^{2}\left(M,\mu\right)\ni f\longmapsto\int_{M}f\,dW^{\left(\mu\right)}\in L^{2}\left(\Omega,Cyl,\mathbb{P}\right)
\end{equation}
as an isometry. Here $\left(\Omega,Cyl,\mathbb{P}\right)$ denotes
the standard probability space, with $Cyl$ abbreviation for the cylinder
sigma-algebra of subsets of $\Omega:=\mathbb{R}^{M}$. For finite
subsets $\left(s_{1},s_{2},\cdots,s_{k}\right)$ in $M$, and Borel
subsets $B_{k}$ in $\mathbb{R}^{k}$, the corresponding cylinder
set 
\[
Cyl\left(\left(s_{i}\right)_{i=1}^{k}\right):=\left\{ \omega\in\mathbb{R}^{M}\mathrel{;}\left(\omega\left(s_{1}\right),\cdots,\omega\left(s_{k}\right)\right)\in B_{k}\right\} .
\]

In summary, we get the the following diagram of isometries, corresponding
to a fixed $\left(\left(k_{x}\right)_{x\in X},\mu\right)\in\mathscr{F}\left(K\right)$,
where $K$ is a fixed p.d. function on $X\times X$:

\begin{figure}[H]
\[
\xymatrix{\mathscr{H}\left(K\right)\ar@/^{1.5pc}/[rr]^{T_{\mu}}\ar@/_{1.2pc}/[dr]_{\text{composition}} &  & L^{2}\left(M,\mu\right)\ar@/^{1.2pc}/[dl]^{\text{Ito-isometry for \ensuremath{W^{\left(\mu\right)}}}}\\
 & L^{2}\left(\Omega,\mathbb{P}\right)
}
\]

\caption{The two isometries. Factorizations by isometries.}

\end{figure}
\end{rem}

\section{\label{sec:exa}Examples and applications}

Below we present four examples in order to illustrate the technical
points in \thmref{E1}. In the first example $X=\left[0,1\right]$,
the unit interval, and in the next two examples $X=\mathbb{D}=\left\{ z\in\mathbb{C}\mathrel{;}\left|z\right|<1\right\} $
the open complex disk. In the fourth example, the Drury-Arveson kernel,
we have $X=\mathbb{C}^{k}$.

We begin with a note on identifications: For $t\in\left[0,1\right]$,
we set 
\[
e\left(t\right):=e^{i2\pi t}.
\]
We write $\lambda_{1}$ for the Lebesgue measure restricted to $\left[0,1\right]$;
and we make the identification: 
\begin{equation}
\left[0,1\right]\cong\mathbb{R}/\mathbb{Z}\cong\mathbb{T}^{1}=\left\{ z\in\mathbb{C}\mathrel{;}\left|z\right|=1\right\} .\label{F1}
\end{equation}
Hence, for $L^{2}\left(\left[0,1\right],\lambda_{1}\right)$ we have
the familiar Fourier expansion: With 
\begin{equation}
f\in L^{2}\left(\lambda_{1}\right),\;\text{and}\;c_{n}:=\int_{\left[0,1\right]}\overline{e\left(nt\right)}f\left(t\right)d\lambda_{1}\left(t\right),\;n\in\mathbb{Z},
\end{equation}
\begin{equation}
f\left(t\right)=\sum_{n\in\mathbb{Z}}c_{n}e\left(nt\right),\;\text{and}\;\int_{0}^{1}\left|f\left(t\right)\right|^{2}d\lambda_{1}\left(t\right)=\sum_{n\in\mathbb{Z}}\left|c_{n}\right|^{2}.
\end{equation}

On $\left[0,1\right]$, we shall also consider the Cantor measure
$\mu_{4}$ with support equal to the Cantor set 
\[
C_{4}=\left\{ x\mathrel{;}x=\sum\nolimits _{k=1}^{\infty}b^{k}/4^{k},\;b_{k}\in\left\{ 0,2\right\} \right\} \subseteq\left[0,1\right];
\]
see \figref{F1} and \cite{MR1655831,jorgensen2018harmonic}.

\begin{figure}
\includegraphics[width=0.35\textwidth]{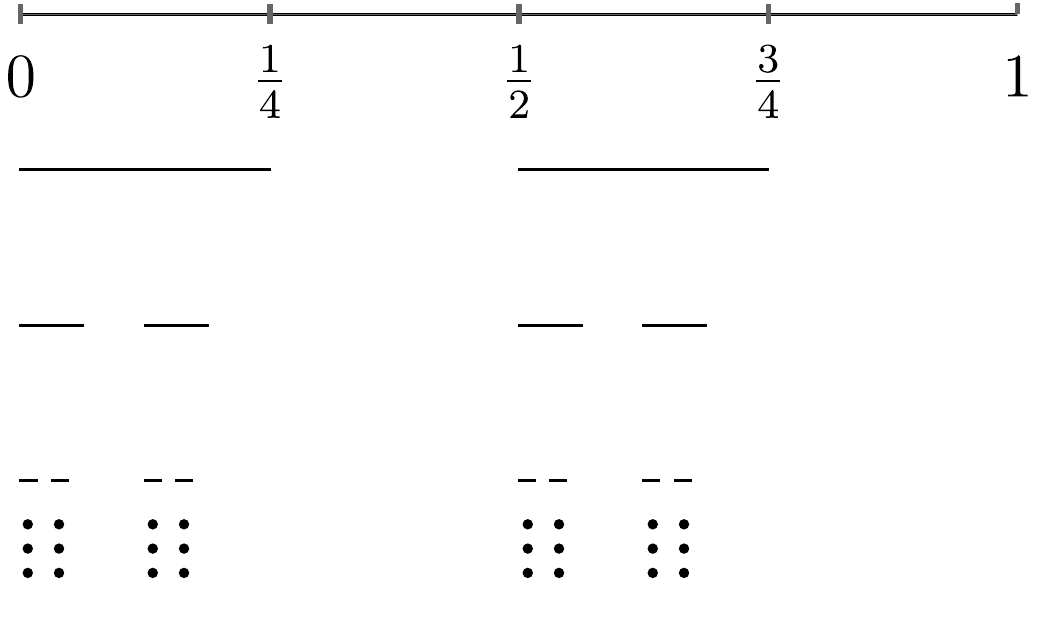}

\caption{\label{fig:F1} The $4$-Cantor set with double gaps as an iterated
function system. This is an iterated-function system construction:
Cantor-set and measure; see (\ref{eq:F4}) below.}

\end{figure}

It is known that $\mu_{4}$ is the unique probability measure s.t.
\begin{equation}
\frac{1}{2}\int_{0}^{1}\left(f\left(\frac{x}{4}\right)+f\left(\frac{x+2}{4}\right)\right)d\mu_{4}\left(x\right)=\int_{0}^{1}f\,d\mu_{4}.\label{eq:F4}
\end{equation}
For the Fourier transform $\widehat{\mu}_{4}$ we have 
\begin{equation}
\widehat{\mu}_{4}\left(t\right)=\prod_{k=1}^{\infty}\frac{1}{2}\left(1+e^{i\pi t/4^{k}}\right),\;t\in\mathbb{R}.
\end{equation}

In \tabref{F1}, we summarize the three examples with the data from
\thmref{E1}. We now turn to the details of the respective examples:

\renewcommand{\arraystretch}{1.5}

\begin{table}
\begin{tabular}{|c|c|c|>{\centering}p{0.35\columnwidth}|c|}
\hline 
 & $X$ & $K$ & $k_{x},M=\left[0,1\right]\cong\mathbb{T}^{1}$, $\mathscr{F}\left(K\right)=\left\{ \left(k_{x},\mu\right)\right\} $ & \multicolumn{1}{c|}{$\mu$}\tabularnewline
\hline 
Ex 1 & $\left[0,1\right]$ & $x\wedge y$ & $k_{x}\left(s\right)=\chi_{\left[0,x\right]}\left(s\right)$ & $\lambda_{1}$\tabularnewline
\hline 
Ex 2 & $\mathbb{D}$ & ${\displaystyle \frac{1}{1-z\overline{w}}}$ & ${\displaystyle k_{z}\left(t\right)=\frac{1}{1-z\overline{e\left(t\right)}}}$ & $\lambda_{1}$ on $\mathbb{T}^{1}$\tabularnewline
\hline 
Ex 3 & $\mathbb{D}$ & ${\displaystyle \prod_{n=0}^{\infty}\left(1+z^{4^{n}}\overline{w}^{4^{n}}\right)}$ & ${\displaystyle k_{z}\left(t\right)=\prod_{n=0}^{\infty}\left(1+z^{4^{n}}\overline{e\left(4^{n}t\right)}\right)}$ & $\mu_{4}$\tabularnewline
\hline 
\multicolumn{1}{c}{} & \multicolumn{1}{c}{} & \multicolumn{1}{c}{} & \multicolumn{1}{>{\centering}p{0.35\columnwidth}}{} & \multicolumn{1}{c}{}\tabularnewline
\end{tabular}

\caption{\label{tab:F1} Three p.d. kernels and their respective Gaussian realizations.}

\end{table}

\renewcommand{\arraystretch}{1}
\begin{example}
\label{exa:F1}If $K\left(x,y\right):=x\wedge y$ is considered a
kernel on $\left[0,1\right]\times\left[0,1\right]$, then the corresponding
RKHS $\mathscr{H}\left(K\right)$ is the Hilbert space of functions
$f$ on $\left[0,1\right]$ such that the distribution derivative
$f'=df/dx$ is in $L^{2}\left(\left[0,1\right],\lambda_{1}\right)$,
$\lambda_{1}=dx$, $f\left(0\right)=0$, and 
\begin{equation}
\left\Vert f\right\Vert _{\mathscr{H}\left(K\right)}^{2}:=\int_{0}^{1}\left|f'\left(x\right)\right|^{2}dx;\label{eq:F6}
\end{equation}
and it is immediate that $\left(k_{x},\lambda_{1}\right)\in\mathscr{F}\left(K\right)$
where $k_{x}\left(s\right):=\chi_{\left[0,x\right]}\left(s\right)$,
the indicator function; see \figref{F2}.

\begin{figure}[H]
\begin{tabular}{c}
\includegraphics[width=0.35\textwidth]{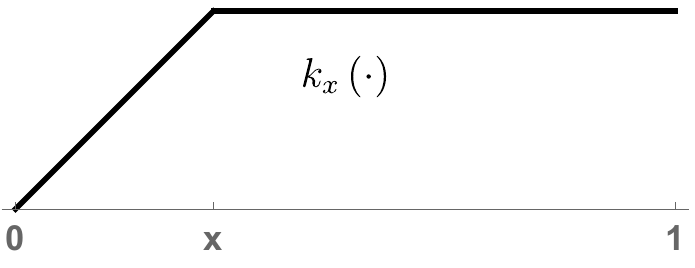}\tabularnewline
\includegraphics[width=0.35\textwidth]{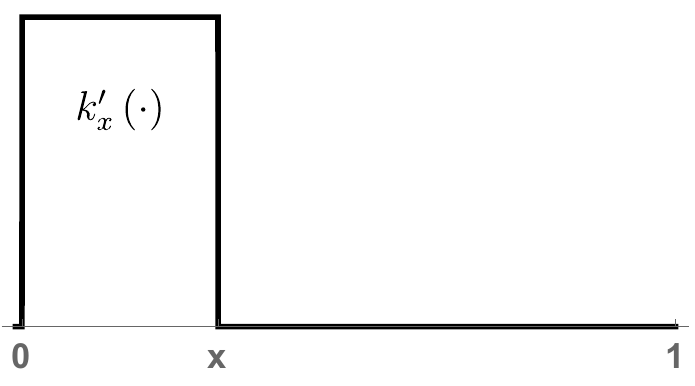}\tabularnewline
\tabularnewline
\end{tabular}

\caption{\label{fig:F2} The generators of the Cameron-Martin RKHS. See \exaref{F1}.}

\end{figure}

\begin{figure}[H]
\includegraphics[width=0.65\textwidth]{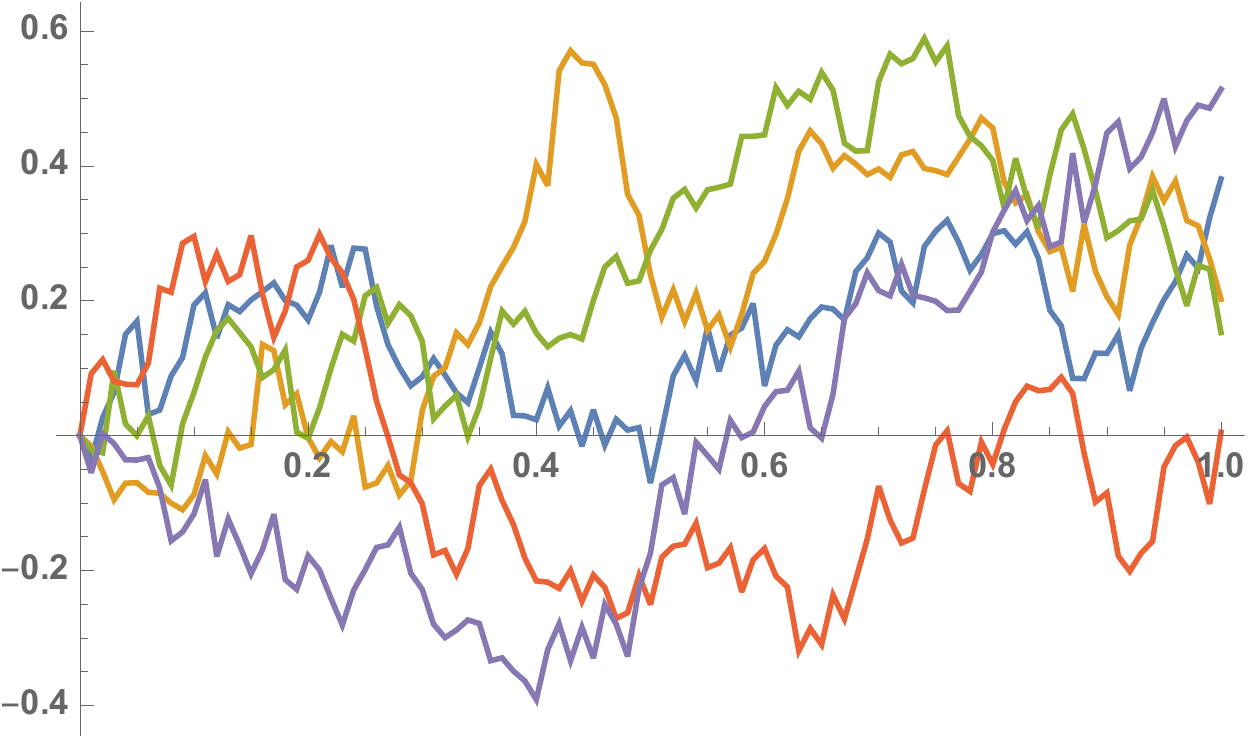}

\caption{\label{fig:F3}Brownian motion on $\left[0,1\right]$. Sample-paths
by Monte Carlo. See \exaref{F1}.}
\end{figure}

\begin{figure}[H]
\includegraphics[width=0.65\textwidth]{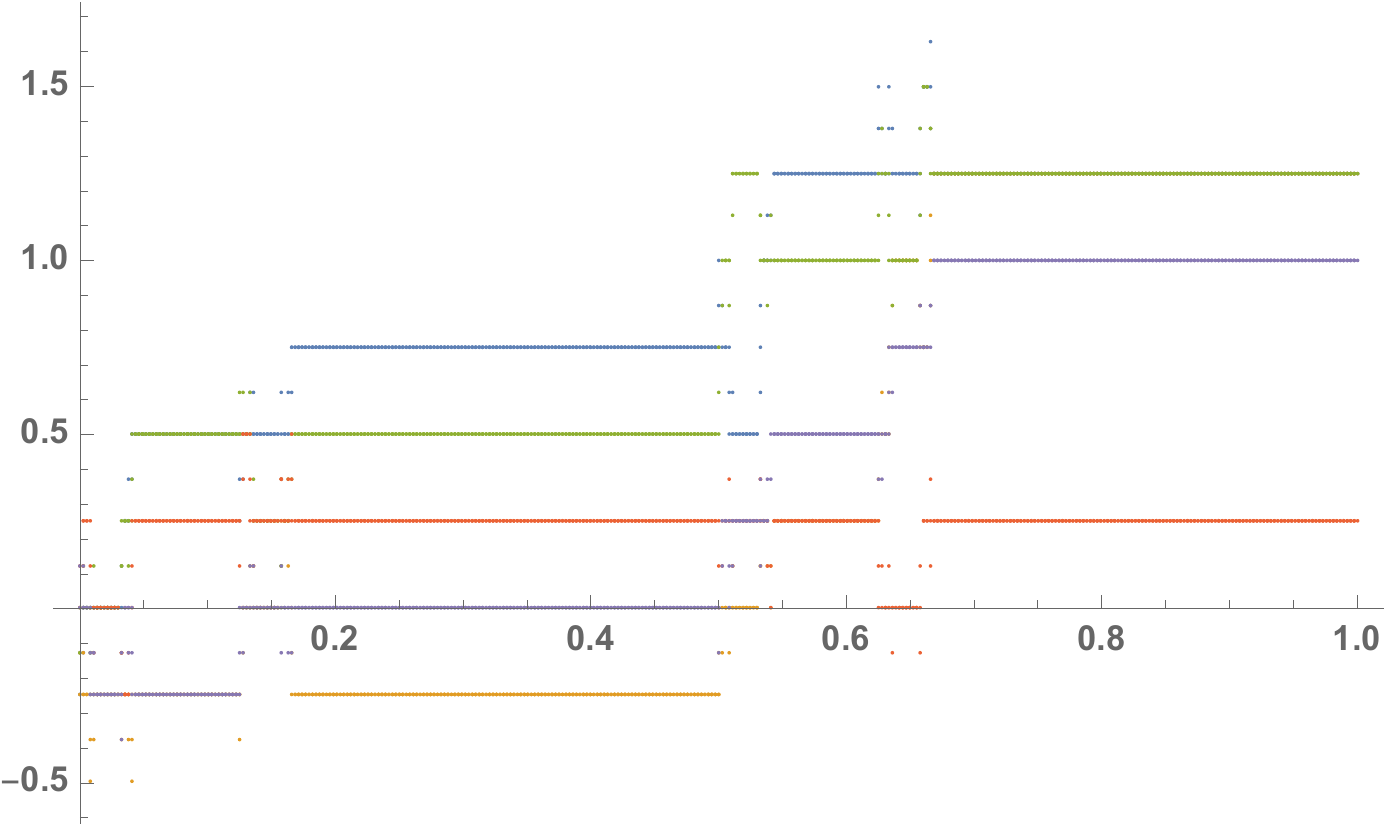}

\caption{\label{fig:F4}A Wiener process with holding patterns in the gaps
of the Cantor set $C_{4}$ in \figref{F1}: The $W^{\left(\mu_{4}\right)}$
process on $\left[0,1\right]$. Sample-paths by Monte Carlo.}
\end{figure}

The process $W^{\left(\lambda_{1}\right)}$ is of course the standard
Brownian motion on $\left[0,1\right]$, pinned at $x=0$; see \figref{F3},
and compare with the $W^{\left(\mu_{4}\right)}$-process in \figref{F4}.
For Monte Carlo simulation, see e.g. \cite{KDB,MR3884672}.

The Hilbert space characterized by (\ref{eq:F6}) is called the Cameron-Martin
space, see e.g., \cite{MR562914}. Moreover, to see that (\ref{eq:F6})
is indeed the precise characterization of the RKHS for this kernel,
one again applies \lemref{B1}.
\end{example}

It immediately follows from \thmref{E1} then the Gaussian processes
corresponding to the data in \tabref{F1} are as follows:
\begin{example}
$z\in\mathbb{D}$: 
\begin{equation}
V_{z}=\int_{0}^{1}\frac{1}{1-z\overline{e\left(t\right)}}dW_{t}^{\left(\lambda_{1}\right)}
\end{equation}
realized as an Ito-integral.

As an application of \thmref{E1}, we get: 
\[
\mathbb{E}\left(\overline{V}_{z}V_{w}\right)=\frac{1}{1-z\overline{w}},\;\forall\left(z,w\right)\in\mathbb{D}\times\mathbb{D}.
\]
\end{example}

\begin{example}
\label{exa:F3}$z\in\mathbb{D}$: 
\begin{equation}
V_{z}=\int_{0}^{1}\prod_{n=0}^{\infty}\left(1+z^{4^{n}}\overline{e\left(4^{n}t\right)}\right)dW_{t}^{\left(\mu_{4}\right)}
\end{equation}
were the $W^{\left(\mu_{4}\right)}$-Ito integral is supported on
the Cantor set $C_{4}\subset\left[0,1\right]$, see \figref{F1}.

As an application of \thmref{E1}, we get: 
\[
\mathbb{E}\left(\overline{V}_{z}V_{w}\right)=\prod_{n=0}^{\infty}\left(1+\left(z\overline{w}\right)^{4^{n}}\right).
\]
\end{example}

The reasoning of \exaref{F3} is based on a theorem of the paper \cite{MR1655831}
(see also \cite{jorgensen2018harmonic}). Set 
\begin{align}
\Lambda_{4} & =\left\{ 0,1,4,5,16,17,20,21,64,65,\cdots\right\} \nonumber \\
 & =\left\{ \sum\nolimits _{k=0}^{\text{finite}}\alpha_{k}4^{k}\mid\alpha_{k}\in\left\{ 0,1\right\} ,\;\text{fintie summations}\right\} \label{eq:F9}
\end{align}
then the Fourier functions $\left\{ e\left(\lambda t\right)\mathrel{;}\lambda\in\Lambda_{4}\right\} $
forms an orthonormal basis in $L^{2}\left(C_{4},\mu_{4}\right)$,
i.e., every $f\in L^{2}\left(C_{4},\mu_{4}\right)$ has its Fourier
expansion 
\begin{align*}
\widehat{f}\left(\lambda\right) & =\int_{C_{4}}\overline{e\left(\lambda t\right)}f\left(t\right)d\mu_{4}\left(t\right);\\
f\left(t\right) & =\sum_{\lambda\in\Lambda_{4}}\widehat{f}\left(\lambda\right)e\left(\lambda t\right);
\end{align*}
and 
\[
\int_{C_{4}}\left|f\right|^{2}d\mu_{4}=\sum_{\lambda\in\Lambda_{4}}\left|\widehat{f}\left(\lambda\right)\right|^{2}.
\]

\begin{lem}
\label{lem:F4}Consider the set $\Lambda_{4}$ in (\ref{eq:F9}),
and, for $s\in\mathbb{D}$, let 
\begin{equation}
F\left(s\right):=\sum_{\lambda\in\Lambda_{4}}s^{\lambda}
\end{equation}
be the corresponding generating function. Then we have the following
infinite-product representation
\begin{equation}
F\left(s\right)=\prod_{n=0}^{\infty}\left(1+s^{4^{n}}\right).\label{eq:F11}
\end{equation}
\end{lem}

\begin{proof}
From (\ref{eq:F9}) we have the following self-similarity for $\Lambda_{4}$:
It is the following identity of sets 
\begin{equation}
\Lambda_{4}=\left\{ 0,1\right\} +4\Lambda_{4}.\label{eq:F12}
\end{equation}
Note that (\ref{eq:F12}) is an algorithm for generating points in
$\Lambda_{4}$. Hence,
\begin{align*}
F\left(s\right) & =\sum_{\lambda\in\Lambda_{4}}s^{\lambda}=\sum_{\left\{ 0,1\right\} +4\Lambda_{4}}s^{\lambda}\\
 & =\sum_{4\Lambda_{4}}s^{\lambda}+s\sum_{4\Lambda_{4}}s^{\lambda}\\
 & =\left(1+s\right)F\left(s^{4}\right)\\
 & =\big(1+s\big)\big(1+s^{4}\big)\cdots\big(1+s^{4^{n-1}}\big)F\big(s^{4^{n}}\big)
\end{align*}
and by induction.

Hence, if $s\in\mathbb{D}$, the infinite-product is absolutely convergent,
and the desired product formula (\ref{eq:F11}) follows.
\end{proof}
\begin{rem}
Note that, in combination with the theorem from \cite{MR1655831}
(see also \cite{jorgensen2018harmonic}), this property of the generating
function $F=F_{\Lambda_{4}}$ from \lemref{F4} is used in the derivation
of the assertions made about the factorization properties in \exaref{F3};
this includes the two formulas (Ex 3) as stated in \tabref{F1}; as
well as of the verification that $\left(k_{z},\mu_{4}\right)\in\mathscr{F}\left(K\right)$,
where $k_{z}$, $\mu_{4}$, and $K$ are as stated.

A direct computation of the two cases, \exaref{F1} and \exaref{F3},
is of interest. Our result, \lemref{D3}, is useful in the construction:
When computing the two Wiener processes $W^{\left(\lambda_{1}\right)}$
and $W^{\left(\mu\right)}$ one notes that the covariance computed
on intervals $\left[0,x\right]$ as $0<x<1$ are as follows:
\begin{align}
\mathbb{E}\left(\left(W_{\left[0,x\right]}^{\left(\lambda_{1}\right)}\right)^{2}\right) & =\lambda_{1}\left(\left[0,x\right]\right)=x,\;\text{and}\label{eq:F13}\\
\mathbb{E}\left(\left(W_{\left[0,x\right]}^{\left(\mu_{4}\right)}\right)^{2}\right) & =\mu_{4}\left(\left[0,x\right]\right).\label{eq:F14}
\end{align}
So the two functions have the representations as in \figref{F5}.

\begin{figure}[H]
\begin{tabular}{>{\centering}p{0.45\columnwidth}>{\centering}p{0.45\columnwidth}}
\includegraphics[width=0.3\textwidth]{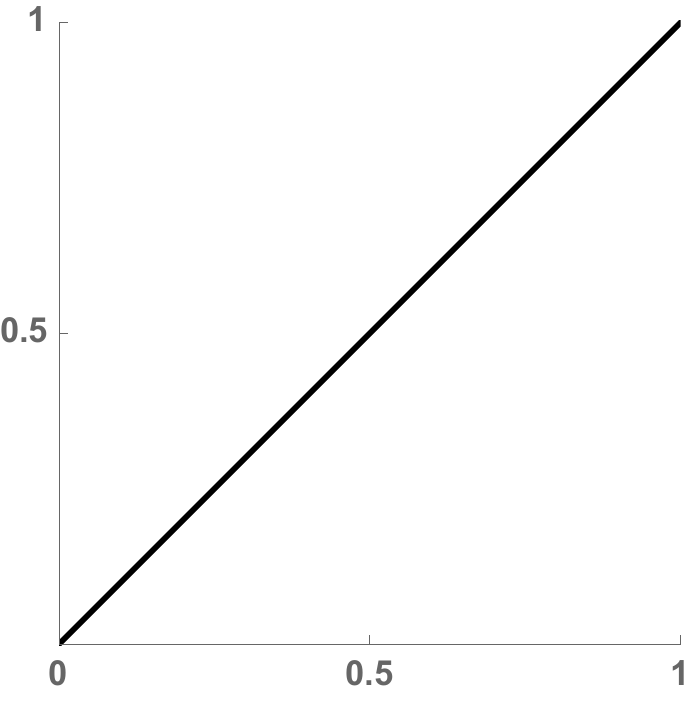} & \includegraphics[width=0.3\textwidth]{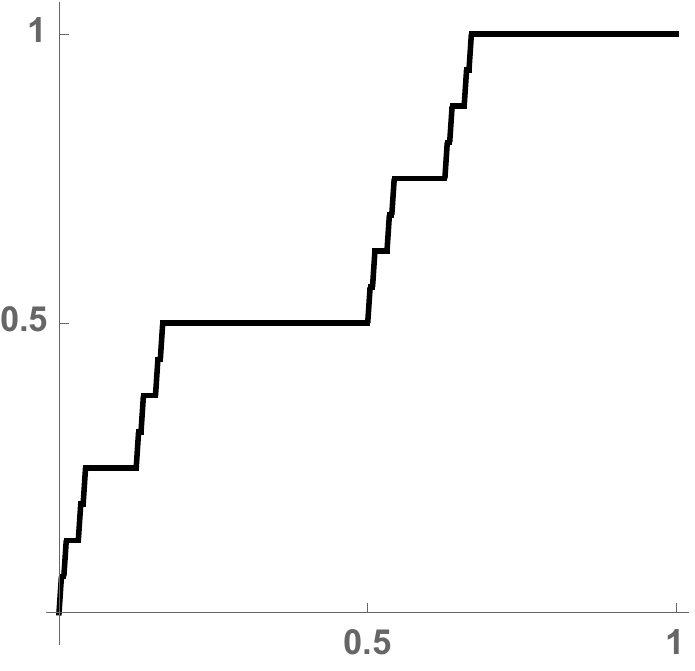}\tabularnewline
The variance formula in (\ref{eq:F13}). & The Devil's staircase. The variance formula in (\ref{eq:F14}).\tabularnewline
\end{tabular}

\caption{\label{fig:F5}The two cumulative distributions.}

\end{figure}
\end{rem}

\begin{example}
The following example illustrates the need for a distinction between
$X$, and families of choices $M$ in \thmref{E1}. \emph{A priori},
one might expect that if $X\times X\xrightarrow{\;K\;}\mathbb{C}$
is given and p.d., it would be natural to try to equip $X$ with a
$\sigma$-algebra $\mathscr{F}_{X}$ of subsets, and a measure $\mu$
such that the condition in (\ref{eq:D2}) holds for $\left(X,\mathscr{F}_{X},\mu\right)$,
i.e., 
\begin{equation}
K\left(x,y\right)=\int_{X}\overline{k}_{x}k_{y}d\mu,\;\left(x,y\right)\in X\times X
\end{equation}
with $\left\{ k_{x}\right\} _{x\in X}$ a system in $L^{2}\left(X,\mathscr{F}_{X},\mu\right)$.
It turns out that there are interesting examples where this is known
to \emph{not }be feasible. The best known such example is perhaps
the Drury-Arveson kernel; see \cite{MR1668582} and \cite{MR2419381,MR2648865}.

Specifics. Consider $\mathbb{C}^{k}$ for $k\geq2$, and $B_{k}\subset\mathbb{C}^{k}$
the complex ball defined for $z=\left(z_{1},\cdots,z_{k}\right)\in\mathbb{C}^{k}$,
\begin{equation}
B_{k}:=\Big\{ z\in\mathbb{C}^{k}\mathrel{;}\underset{\left\Vert z\right\Vert _{2}^{2}}{\underbrace{\sum\nolimits _{j=1}^{k}\left|z_{j}\right|^{2}}}<1\Big\}.
\end{equation}
For $z,w\in\mathbb{C}^{k}$, set 
\begin{align}
\left\langle z,w\right\rangle  & :=\sum_{j=1}^{k}z_{j}\overline{w}_{j},\;\text{and}\nonumber \\
K_{DA}\left(z,w\right) & :=\frac{1}{1-\left\langle z,w\right\rangle },\;\left(z,w\right)\in B_{k}\times B_{k}.\label{eq:F17}
\end{align}
\end{example}

\begin{cor}[{Arveson \cite[Coroll 2]{MR1668582}}]
\label{cor:F7}Let $k\geq2$, and let $\mathscr{H}\left(K_{DA}\right)$
be the RKHS of the D-A kernel in (\ref{eq:F17}). Then there is \uline{no}
Borel measure on $\mathbb{C}^{k}$ such that $\left(\mathbb{C}^{k},\mathscr{B}_{k},\mu\right)\in\mathscr{F}\left(K_{DA}\right)$;
i.e., there is \uline{no} solution to the formula 
\[
\left\Vert f\right\Vert _{\mathscr{H}\left(K_{DA}\right)}^{2}=\int_{\mathbb{C}^{k}}\left|f\left(z\right)\right|^{2}d\mu\left(z\right),
\]
for all $f\left(z\right)$ $k$-polynomials.
\end{cor}

\begin{rem}
It is natural to ask about disintegration properties for the Gaussian
process $V_{DA}$ corresponding to the Drury-Arveson kernel (\ref{eq:F17}).
Combining our \thmref{E1} above with the corollary (Coroll \ref{cor:F7}),
we conclude that, in two or more complex dimensions $k$, the question
of finding the admissible disintegrations this Gaussian process $V_{DA}$
is subtle. It must necessarily involve measure spaces going beyond
$\mathbb{C}^{k}$.
\end{rem}

\section{\label{sec:atomic}The case of $\left(k_{x},\mu\right)\in\mathscr{F}\left(K\right)$
when $\mu$ is atomic}

Below we present a case where $\mu$ from pairs in $\mathscr{F}\left(K\right)$
may be chosen to be atomic. The construction is general, but for the
sake of simplicity we shall assume that a given p.d. $K$ is such
that the RKHS $\mathscr{H}\left(K\right)$ is separable, i.e., when
it has an (all) orthonormal basis (ONB) indexed by $\mathbb{N}$.
\begin{defn}
\label{def:G1}Let $\mathscr{H}$ be a Hilbert space (separable),
and let $\left\{ g_{n}\right\} _{n\in\mathbb{N}}$ be a system of
vectors in $\mathscr{H}$ such that 
\begin{equation}
\sum_{n\in\mathbb{N}}\left|\left\langle \psi,g_{n}\right\rangle _{\mathscr{H}}\right|^{2}=\left\Vert \psi\right\Vert _{\mathscr{H}}^{2}\label{eq:F1}
\end{equation}
holds for all $\psi\in\mathscr{H}$. We then say that $\left\{ g_{n}\right\} _{n\in\mathbb{N}}$
is a Parseval frame for $\mathscr{H}$. (Also see \defref{J1}.)

An equivalent assumption is that the mapping 
\begin{equation}
\mathscr{H}\ni\psi\xmapsto{\quad T\quad}\left(\left\langle \psi,g_{n}\right\rangle _{\mathscr{H}}\right)\in l^{2}\left(\mathbb{N}\right)
\end{equation}
is isometric. One checks that then the adjoint $T^{*}:l^{2}\rightarrow\mathscr{H}$
is: 
\[
T^{*}\left(\left(\xi_{n}\right)\right)=\sum_{n\in\mathbb{N}}\xi_{n}g_{n}\in\mathscr{H}.
\]
\end{defn}

For general background references on frames in Hilbert space, we refer
to \cite{MR2367342,MR2538596,MR3118429,MR3005286,MR3009685,MR3204026,MR3121682,MR3275625,MR3574374},
and also see \cite{MR3005286,MR3526434,MR3688637,MR3700114,MR3800275}.
\begin{lem}
\label{lem:G2}Let $K$ be given p.d. on $X\times X$, and assume
that $\left\{ g_{n}\right\} _{n\in\mathbb{N}}$ is a Parseval frame
in $\mathscr{H}\left(K\right)$; then 
\begin{equation}
K\left(x,y\right)=\sum_{n\in\mathbb{N}}g_{n}\left(x\right)\overline{g_{n}\left(y\right)}\label{eq:G3}
\end{equation}
with the sum on the RHS in (\ref{eq:G3}) absolutely convergent.
\end{lem}

\begin{proof}
By the reproducing property of $\mathscr{H}\left(K\right)$, see \secref{pdk},
we get, for all $\left(x,y\right)\in X\times X$: 
\begin{eqnarray*}
K\left(x,y\right) & = & \left\langle K\left(\cdot,x\right),K\left(\cdot,y\right)\right\rangle _{\mathscr{H}\left(K\right)}\\
 & \underset{\text{by \ensuremath{\left(\ref{eq:F1}\right)}}}{=} & \sum_{n\in\mathbb{N}}\left\langle K\left(\cdot,x\right),g_{n}\right\rangle _{\mathscr{H}\left(K\right)}\left\langle g_{n},K\left(\cdot,y\right)\right\rangle _{\mathscr{H}\left(K\right)}\\
 & \underset{\text{by \ensuremath{\left(\ref{eq:A4}\right)}}}{=} & \sum_{n\in\mathbb{N}}g_{n}\left(x\right)\overline{g_{n}\left(y\right)}.
\end{eqnarray*}
\end{proof}
Now a direct application of the argument in the proof of \thmref{E1}
yields the following:
\begin{cor}
Let $K$ be given p.d. on $X\times X$ such that $\mathscr{H}\left(K\right)$
is separable, and let $\left\{ g_{n}\right\} _{n\in\mathbb{N}}$ be
a Parseval frame, for example an ONB in $\mathscr{H}\left(K\right)$.
Let $\left\{ \zeta_{n}\right\} _{n\in\mathbb{N}}$ be a chosen system
of i.i.d. (independent identically distributed) system of standard
Gaussians, i.e., with $N\left(0,1\right)$-distribution $\nicefrac{1}{\sqrt{2\pi}}e^{\nicefrac{-s^{2}}{2}}$,
$s\in\mathbb{R}$. Then the following sum defines a Gaussian process,
\begin{equation}
V_{x}\left(\cdot\right):=\sum_{n\in\mathbb{N}}g_{n}\left(x\right)\zeta_{n}\left(\cdot\right),
\end{equation}
i.e., $\left\{ V_{x}\right\} _{x\in X}$ is well-defined in $L^{2}\left(\Omega,Cyl,\mathbb{P}\right)$,
as stated, where $\Omega=\mathbb{R}^{\mathbb{N}}$ as a realization
in an infinite Cartesian product with the usual cylinder $\sigma$-algebra,
and $\left\{ V_{x}\right\} _{x\in X}$ has $K$ as covariance kernel,
i.e., 
\[
\mathbb{E}\left(\overline{V}_{x}V_{y}\right)=K\left(x,y\right),\;\forall\left(x,y\right)\in X\times X;
\]
see (\ref{eq:D11}).
\end{cor}

\begin{proof}
This is a direct application of \lemref{G2}, and we leave the remaining
verifications to the reader.
\end{proof}

\section{Point processes: The case when $\left\{ \delta_{x}\right\} \subset\mathscr{H}\left(K\right)$}

Let $X\times X\xrightarrow{\;K\;}\mathbb{R}$ be a fixed positive
definite kernel. We know that the RKHS $\mathscr{H}\left(K\right)$
consists of functions $h$ on $X$ subject to the \emph{a priori}
estimate in \lemref{B1}. For recent work on point-processes over
infinite networks \cite{MR3860446,MR3246982,MR3843552,MR3670916,MR3390972,MR3861681,MR3848251,MR3881668,MR3854505},
the case when the Dirac measures $\delta_{x}$ are in $\mathscr{H}\left(K\right)$
is of special significance. In this case there is an abstract Laplace
operator $\Delta$, defined as follows:
\begin{equation}
\left(\Delta h\right)\left(x\right)=\left\langle \delta_{x},h\right\rangle _{\mathscr{H}\left(K\right)},\;\forall h\in\mathscr{H}\left(K\right).\label{eq:H1}
\end{equation}
For the $\left\Vert \cdot\right\Vert _{\mathscr{H}\left(K\right)}$-norm
of $\delta_{x}$, we have
\begin{equation}
\left(\Delta\delta_{x}\right)\left(x\right)=\left\Vert \delta_{x}\right\Vert _{\mathscr{H}\left(K\right)}^{2};\label{eq:H2}
\end{equation}
immediate from (\ref{eq:H1}).

For every finite subset $F\subset X$, we consider the induced $\left|F\right|\times\left|F\right|$
matrix 
\begin{equation}
K_{F}\left(x,y\right)=\left(K\left(x,y\right)\right)_{x,y\in F}.
\end{equation}
Note that $K_{F}$ is a positive definite square matrix. Its spectrum
consists of eigenvalues $\lambda_{s}\left(F\right)$.

If $\left(K,X\right)$ is as described, i.e., $X\times X\xrightarrow{\;K\;}\mathbb{R}\left(\text{or }\mathbb{C}\right)$
p.d., and if 
\begin{equation}
\left\{ \delta_{x}\right\} _{x\in X}\subset\mathscr{H}\left(K\right),\label{eq:H4-1}
\end{equation}
we shall see that $X$ must then be discrete. (In interesting cases,
also countable.) If (\ref{eq:H4-1}) holds, we shall say that $\left(K,X\right)$
is a \emph{point process}. We shall further show that point processes
arise by restriction as follows:

Let $\left(K,X\right)$ be given with $K$ a p.d. kernel. If a countable
subset $S\subset X$ is such that $K^{\left(S\right)}:=K\big|_{S\times S}$
has 
\begin{equation}
\left\{ \delta_{x}\right\} _{x\in S}\in\mathscr{H}(K^{\left(S\right)}),
\end{equation}
then we shall say that $\left(K^{\left(S\right)},S\right)$ is an
\emph{induced point process}.

\subsection{Nets of finite submatrices, and their limits}

Given $(K,X)$ as above with $K$ p.d. and defined on $X\times X$.
Then the finite submatrices in the subsection header are indexed by
the net of all finite subsets $F$ of $X$ as follows: Given $F$,
then the corresponding $\left|F\right|\times\left|F\right|$ square
matrix $K_{F}$ is simply the restriction of $K$ to $F\times F$.
Of course, each matrix $K_{F}$ is positive definite, and so it has
a finite list of eigenvalues. These eigenvalue lists figure in the
discussion below.
\begin{lem}
\label{lem:H1}Let $K$, $F$, and $K_{F}$ be as above, with $\lambda_{s}\left(F\right)$
denoting the numbers in the list of eigenvalues for the matrix $K_{F}$.
Then
\begin{equation}
1\leq\lambda_{s}\left(F\right)\sum_{x\in F}\left\Vert \delta_{x}\right\Vert _{\mathscr{H}\left(K\right)}^{2}.\label{eq:H4}
\end{equation}
\end{lem}

\begin{proof}
Consider the eigenvalue equation 
\begin{equation}
\left(\xi_{x}\right)_{x\in F},\quad\sum_{x\in F}\left|\xi_{x}\right|^{2}=\left\Vert \xi\right\Vert _{2}^{2}=1,\quad K_{F}\xi=\lambda_{s}\left(F\right)\xi.
\end{equation}
From \lemref{B1} and for $x\in F$, we then get
\begin{align}
\left|\xi_{x}\right|^{2} & \leq\left\Vert \delta_{x}\right\Vert _{\mathscr{H}\left(K\right)}^{2}\left\langle \xi,K_{F}\xi\right\rangle _{l^{2}\left(F\right)}\nonumber \\
 & =\left\Vert \delta_{x}\right\Vert _{\mathscr{H}\left(K\right)}^{2}\lambda_{s}\left(F\right).\label{eq:H6}
\end{align}
Now apply $\sum_{x\in F}$ to both sides in (\ref{eq:H6}), and the
desired conclusion (\ref{eq:H4}) follows.
\end{proof}
\begin{rem}
A consequence of the lemma is that the matrices $K_{F}^{-1}$ and
$K_{F}^{-1/2}$ automatically are well defined (by the spectral theorem)
with associated spectral bounds.
\end{rem}

\begin{defn}
Let $K$, $F$, and $K_{F}$ be as above; and with the condition $\delta_{x}\in\mathscr{H}\left(K\right)$
in force. Set 
\begin{equation}
\mathscr{H}_{K}\left(F\right):=span_{x\in F}\left\{ K\left(\cdot,x\right)\right\} .
\end{equation}
\end{defn}

It is a finite-dimensional (and therefore closed) subspace in $\mathscr{H}\left(K\right)$.
The orthogonal projection onto $\mathscr{H}_{K}\left(F\right)$ will
be denoted $P_{F}:\mathscr{H}\left(K\right)\rightarrow\mathscr{H}_{K}\left(F\right)$.
\begin{lem}
\label{lem:H4}Let $K$, $F$, $K_{F}$, and $\mathscr{H}_{K}\left(F\right)$
be as above. Then the orthogonal projection $P_{F}$ is as follows:
For $h\in\mathscr{H}\left(K\right)$, set $h_{F}=h\big|_{F}$, restriction:
\begin{equation}
\left(P_{F}h\right)\left(\cdot\right)=\sum_{y\in F}\left(K_{F}^{-1}h_{F}\right)\left(y\right)K\left(\cdot,y\right).\label{eq:H8}
\end{equation}
\end{lem}

\begin{proof}
It is immediate from the definition that $P_{F}h$ has the form 
\begin{equation}
P_{F}h=\sum_{y\in F}\xi_{y}K\left(\cdot,y\right)
\end{equation}
with $\left(\xi_{y}\right)_{y\in F}\in\mathbb{C}^{\left|F\right|}$.
Since $P_{F}$ is the orthogonal projection, 
\begin{equation}
\left(h-P_{F}h\right)\perp_{\mathscr{H}\left(K\right)}\left\{ K\left(\cdot,y\right)\right\} _{y\in F}
\end{equation}
(orthogonality in the $\mathscr{H}\left(K\right)$-inner product)
which yields:
\[
h\left(x\right)=\left(K_{F}\xi\right)\left(x\right)\left(=\sum\nolimits _{y\in F}K\left(x,y\right)\xi_{y}\right),\;\forall x\in F;
\]
and therefore, $\xi=K_{F}^{-1}h_{F}$, which is the desired formula
(\ref{eq:H8}).
\end{proof}
\begin{cor}
\label{cor:H5}Let $X$, $K$, $\mathscr{H}\left(K\right)$ be as
above, and assume $\delta_{x}\in\mathscr{H}\left(K\right)$ for some
$x\in X$. Then a function $h$ on $X$ is in $\mathscr{H}\left(K\right)$
if and only if 
\begin{equation}
\sup_{F}\left\Vert K_{F}^{-1/2}h_{F}\right\Vert _{l^{2}\left(F\right)}<\infty,\label{eq:H11}
\end{equation}
where the supremum is over all finite subsets $F$ of $X$. If $h$
is finite energy, then
\begin{equation}
\left\Vert h\right\Vert _{\mathscr{H}\left(K\right)}^{2}=\sup_{F}\left\Vert K_{F}^{-1/2}h_{F}\right\Vert _{l^{2}\left(F\right)}^{2}.\label{eq:H12}
\end{equation}
\end{cor}

\begin{proof}
The proof follows from an application of Hilbert space geometry to
the RKHS $\mathscr{H}\left(K\right)$, on the family of orthogonal
projections $P_{F}$ indexed by the finite subsets $F$ in $X$. With
the standard lattice operations, applied to projections, we have $\sup_{F}P_{F}=I_{\mathscr{H}\left(K\right)}$.
The conclusions (\ref{eq:H11})-(\ref{eq:H12}) follow from this since,
by the lemma, 
\begin{eqnarray}
\left\Vert P_{F}h\right\Vert _{\mathscr{H}\left(K\right)}^{2} & \underset{\text{by \ensuremath{\left(\ref{eq:A3}\right)}}}{=} & \left\langle K_{F}^{-1}h_{F},K_{F}K_{F}^{-1}h_{F}\right\rangle _{l^{2}\left(F\right)}\nonumber \\
 & = & \left\langle h_{F},K_{F}^{-1}h_{F}\right\rangle _{l^{2}\left(F\right)}=\left\Vert K_{F}^{-1/2}h_{F}\right\Vert _{l^{2}\left(F\right)}^{2}.\label{eq:H13}
\end{eqnarray}
\end{proof}

\begin{rem}
The advantage with the use of this system of orthogonal projections
$P_{F}$, indexed by the finite subsets $F$ of $X$, is that we may
then take advantage of the known lattice operations for orthogonal
projections in Hilbert space. But it is important that we get approximation
with respect to the canonical norm in the RKHS $\mathscr{H}\left(K\right)$.
This works because by our construction, the orthogonality properties
for the projections $P_{F}$ refers precisely to the inner product
in $\mathscr{H}\left(K\right)$. Naturally we get the best $\mathscr{H}\left(K\right)$-approximation
properties when $X$ is further assumed countable. But the formula
for the $\mathscr{H}\left(K\right)$-norm holds in general.
\end{rem}

\begin{cor}
\label{cor:H7}Let $X\times X\xrightarrow{\;K\;}\mathbb{C}$ be fixed,
assumed p.d., and let $\mathscr{H}\left(K\right)$ be the corresponding
RKHS. Let $x\in X$ be given. Then $\delta_{x}\in\mathscr{H}\left(K\right)$
if and only if 
\begin{equation}
\sup_{F\subset X,\:F\text{ finite, \ensuremath{x\in F}}}\left(K_{F}^{-1}\right)_{x,x}<\infty.\label{eq:H13-1}
\end{equation}
In this case, we have: 
\[
\left\Vert \delta_{x}\right\Vert _{\mathscr{H}\left(K\right)}^{2}=\text{the supremum in }\left(\ref{eq:H13-1}\right).
\]
\end{cor}

\begin{proof}
The result is immediate from \corref{H5} applied to $h:=\delta_{x}$,
where $x$ is fixed. Here the terms in (\ref{eq:H12}) are, for $F$
finite, $x\in F$: 
\begin{equation}
\left\langle \delta_{x}\big|_{F},K_{F}^{-1}\left(\delta_{x}\big|_{F}\right)\right\rangle _{l^{2}\left(F\right)}=\left(K_{F}^{-1}\right)_{x,x},
\end{equation}
and the stated conclusion is now immediate.
\end{proof}
\begin{cor}
\label{cor:H8}Let $X$, $K$, and $\mathscr{H}\left(K\right)$ be
as above, but assume now that $X$ is countable, with a monotone net
of finite sets:
\begin{equation}
F_{1}\subset F_{2}\subset F_{3}\cdots,\;\text{and}\quad X=\cup_{i\in\mathbb{N}}F_{i};
\end{equation}
then a function $h$ on $X$ is in $\mathscr{H}\left(K\right)$ iff
$\sup_{i}\left\Vert K_{F_{i}}^{-1/2}h\big|_{F_{i}}\right\Vert _{l^{2}\left(F_{i}\right)}<\infty$.

Moreover, 
\begin{equation}
\left\Vert h\right\Vert _{\mathscr{H}_{E}}^{2}=\lim_{i\rightarrow\infty}\left\Vert K_{F_{i}}^{-1/2}h\big|_{F_{i}}\right\Vert _{l^{2}\left(F_{i}\right)}^{2},\label{eq:H15}
\end{equation}
where, the convergence in (\ref{eq:H15}) is monotone.
\end{cor}

\begin{proof}
From the definition of the order of orthogonal projections, we have
\begin{equation}
P_{F_{1}}\leq P_{F_{2}}\leq P_{F_{3}}\leq\cdots,
\end{equation}
and therefore, 
\begin{equation}
\left\Vert P_{F_{1}}h\right\Vert _{\mathscr{H}\left(K\right)}^{2}\leq\left\Vert P_{F_{2}}h\right\Vert _{\mathscr{H}\left(K\right)}^{2}\leq\left\Vert P_{F_{3}}h\right\Vert _{\mathscr{H}\left(K\right)}^{2}\leq\cdots,\label{eq:H17}
\end{equation}
with $\lim_{i\rightarrow\infty}\left\Vert P_{F_{i}}h\right\Vert _{\mathscr{H}\left(K\right)}^{2}=\left\Vert h\right\Vert _{\mathscr{H}\left(K\right)}^{2}$.
But by (\ref{eq:H13}) and the proof of \corref{H5}, we have 
\[
\left\Vert K_{F_{i}}^{-1/2}h\big|_{F_{i}}\right\Vert _{l^{2}\left(F_{i}\right)}^{2}=\left\Vert P_{F_{i}}h\right\Vert _{\mathscr{H}\left(K\right)}^{2}
\]
and, so, by (\ref{eq:H17}), we get: 
\[
\left\Vert K_{F_{1}}^{-1/2}h\big|_{F_{1}}\right\Vert _{l^{2}\left(F_{1}\right)}^{2}\leq\left\Vert K_{F_{2}}^{-1/2}h\big|_{F_{2}}\right\Vert _{l^{2}\left(F_{2}\right)}^{2}\leq\left\Vert K_{F_{3}}^{-1/2}h\big|_{F_{3}}\right\Vert _{l^{2}\left(F_{3}\right)}^{2}\leq\cdots.
\]
The conclusion now follows.
\end{proof}

\subsection{Restrictions of p.d. kernels}

Below we shall be considering pairs $(K,X)$ with $K$ a fixed p.d.
kernel defined on $X\times X$, and, as before, we denote by $\mathscr{H}\left(K\right)$
the corresponding RKHS with its canonical inner product. In general,
$X$ is an arbitrary set, typically of large cardinality, in particular
uncountable: It may be a complex domain, a generalized boundary, or
it may be a manifold arising from problems in physics, in signal processing,
or in machine learning models. Moreover, for such general pairs $(K,X)$,
with $K$ a fixed p.d. kernel, the Dirac functions $\delta_{x}$ are
typically not in $\mathscr{H}\left(K\right)$.

Here we shall turn to induced systems, indexed by suitable countable
discrete subsets $S$ of $X$. Indeed, for a number of sampling or
interpolation problems, it is possible to identify countable discrete
subsets $S$ of $X$, such that when $K$ is restricted to $S\times S$,
i.e., $K^{\left(S\right)}:=K\big|_{S\times S}$, then for $x\in S$,
the Dirac functions $\delta_{x}$ will be in $\mathscr{H}\left(K^{\left(S\right)}\right)$;
i.e., we get induced point processes indexed by $S$. In fact, with
\corref{H8}, we will be able to identify a variety of such subsets
$S$.

Moreover, each such choice of subset $S$ yields point-process, and
an induced graph, and graph Laplacian; see (\ref{eq:H1})-(\ref{eq:H2}).
These issues will be taken up in detail in the two subsequent sections.
In the following \exaref{H8}, for illustration, we identify a particular
instance of this, when $X=\mathbb{R}$ (the reals), and $S=\mathbb{Z}$
(the integers), and where $K$ is the covariance kernel of standard
Brownian motion on $\mathbb{R}$.
\begin{example}[\textbf{Discretizing the covariance function for Brownian motion on
$\mathbb{R}$}]
\label{exa:H8}The present example is a variant of \exaref{F1},
but with $X=\mathbb{R}$ (instead of the interval $\left[0,1\right]$).
We now set
\begin{equation}
K\left(x,y\right):=\begin{cases}
\left|x\right|\wedge\left|y\right| & \left(x,y\right)\in\mathbb{R}\times\mathbb{R},\;xy\geq0;\\
0 & xy<0.
\end{cases}\label{eq:H22}
\end{equation}
It is immediate that (\ref{eq:F6}) in \exaref{F1} carries over,
but now with $\mathbb{R}$ in place of $\left[0,1\right]$. The normalization
$f\left(0\right)=0$ is carried over. We get that: A function $f\left(x\right)$
on $\mathbb{R}$ is in $\mathscr{H}\left(K\right)$ iff it has distribution-derivative
$f'=df/dx$ in $L^{2}\left(\mathbb{R}\right)$, see (\ref{eq:H19}).
As before, we conclude that the $\mathscr{H}\left(K\right)$-norm
is: 
\begin{equation}
\left\Vert f\right\Vert _{\mathscr{H}\left(K\right)}^{2}=\int_{\mathbb{R}}\left|f'\right|^{2}dx;\label{eq:H19}
\end{equation}
see also \lemref{D3}.

Set 
\begin{equation}
K^{\left(\mathbb{Z}\right)}=K\big|_{\mathbb{Z}\times\mathbb{Z}},
\end{equation}
and consider the corresponding RKHS $\mathscr{H}\left(K^{\left(\mathbb{Z}\right)}\right)$.
Using \cite{MR3450534,MR3507188}, we conclude that functions $\Phi$
on $\mathbb{Z}$ are in $\mathscr{H}\left(K^{\left(\mathbb{Z}\right)}\right)$
iff $\Phi\left(0\right)=0$, and 
\[
\sum_{n\in\mathbb{Z}}\left|\Phi\left(n\right)-\Phi\left(n+1\right)\right|^{2}<\infty.
\]
In that case, 
\begin{equation}
\left\Vert \Phi\right\Vert _{\mathscr{H}\left(K^{\left(\mathbb{Z}\right)}\right)}^{2}=\sum_{n\in\mathbb{Z}}\left|\Phi\left(n\right)-\Phi\left(n+1\right)\right|^{2}.\label{eq:H24-1}
\end{equation}
For the $\mathbb{Z}$-kernel, we have: $\left\{ \delta_{n}\right\} _{n\in\mathbb{Z}}\subset\mathscr{H}\left(K^{\left(\mathbb{Z}\right)}\right)$,
and 
\begin{equation}
\delta_{n}\left(\cdot\right)=2K\left(\cdot,n\right)-K\left(\cdot,n+1\right)-K\left(\cdot,n-1\right),\;\forall n\in\mathbb{Z}.\label{eq:H26-1}
\end{equation}
Moreover, the corresponding Laplacian $\Delta$ from (\ref{eq:H1})
is 
\begin{equation}
\left(\Delta\Phi\right)\left(n\right)=2\Phi\left(n\right)-\Phi\left(n+1\right)-\Phi\left(n-1\right),
\end{equation}
i.e., the standard discretized Laplacian.

From the matrices $K_{F}^{\left(\mathbb{Z}\right)}$, $F\subset\mathbb{Z}$,
we have the following; illustrated with $F=F_{N}=\left\{ 1,2,\cdots,N\right\} $.
\begin{align}
K_{F_{N}}^{\left(\mathbb{Z}\right)} & =\begin{bmatrix}1 & 1 & 1 & 1 & \cdots & \cdots & 1\\
1 & 2 & 2 & 2 & \cdots & \cdots & 2\\
1 & 2 & 3 & 3 & \cdots & \cdots & 3\\
1 & 2 & 3 & 4 & \cdots & \cdots & 4\\
\vdots & \vdots & \vdots & \vdots & \ddots & \ddots & \vdots\\
\vdots & \vdots & \vdots & \vdots & \ddots & N-1 & N-1\\
1 & 2 & 3 & 4 & \cdots & N-1 & N
\end{bmatrix},\\
\intertext{and}\left(K_{F_{N}}^{\left(\mathbb{Z}\right)}\right)^{-1} & =\begin{bmatrix}2 & -1 & 0 & 0 & 0 & \cdots & 0\\
-1 & 2 & -1 & 0 & 0 & \cdots & 0\\
0 & -1 & 2 & -1 & 0 & \cdots & 0\\
\vdots & \ddots & \ddots & \ddots & \ddots & \ddots & \vdots\\
\vdots &  & \ddots & -1 & 2 & -1 & 0\\
0 & 0 & \cdots & 0 & -1 & 2 & -1\\
0 & 0 & \cdots & 0 & 0 & -1 & 1
\end{bmatrix}.
\end{align}
 In particular, we have for $n,m\in\mathbb{Z}$: 
\[
\left\langle \delta_{n},\delta_{m}\right\rangle _{\mathscr{H}\left(K^{\left(\mathbb{Z}\right)}\right)}=\begin{cases}
2 & \text{if \ensuremath{n=m}}\\
-1 & \text{if \ensuremath{\left|n-m\right|}=1}\\
0 & \text{otherwise}.
\end{cases}
\]
\end{example}

\begin{rem}
The determinant of $K_{F_{N}}^{\left(\mathbb{Z}\right)}$ is 1 for
all $N$. \emph{Proof}. By eliminating the first column, and then
the first row, $\det(K_{F_{N}}^{\left(\mathbb{Z}\right)})$ is reduced
to $\det(K_{F_{N-1}}^{\left(\mathbb{Z}\right)})$ . So by induction,
the determinant is 1.

Note that 
\[
\sum_{k\in\mathbb{Z}}\chi_{\left[1,n\right]}\left(k\right)\chi_{\left[1,m\right]}\left(k\right)=n\wedge m
\]
which yields the factorization 
\begin{equation}
K_{F_{N}}^{\left(\mathbb{Z}\right)}=A_{N}A_{N}^{*},\label{eq:H29-1}
\end{equation}
i.e., 
\[
K_{F_{N}}^{\left(\mathbb{Z}\right)}\left(n,m\right)=\left(A_{N}A_{N}^{*}\right)_{n,m}=\sum A_{N}\left(n,k\right)A_{N}^{*}\left(k,m\right),
\]
where $A_{N}$ is the $N\times N$ lower triangular matrix given by
\[
A_{N}=\begin{bmatrix}1 & 0 & \cdots & \cdots & 0\\
1 & 1 & 0 & \cdots & 0\\
\vdots & \vdots & \ddots & \ddots & \vdots\\
\vdots & \vdots & \vdots & \ddots & 0\\
1 & 1 & \cdots & \cdots & 1
\end{bmatrix}.
\]
In particular, we get that $\det(K_{F_{N}}^{\left(\mathbb{Z}\right)})=1$
immediately. This is a special case of \thmref{E1}.

For the general case, let $F_{N}=\left\{ x_{j}\right\} _{j=1}^{N}$
be a finite subset of $\mathbb{R}$, assuming $x_{1}<x_{2}<\cdots<x_{N}$.
Then the factorization (\ref{eq:H29-1}) holds with 
\begin{equation}
A_{N}=\begin{bmatrix}\sqrt{x_{1}} & 0 & 0 & \cdots & 0\\
\sqrt{x_{1}} & \sqrt{x_{2}-x_{1}} & 0 & \cdots & \vdots\\
\sqrt{x_{1}} & \sqrt{x_{2}-x_{1}} & \sqrt{x_{3}-x_{2}} & \ddots & \vdots\\
\vdots & \vdots & \vdots & \ddots & 0\\
\sqrt{x_{1}} & \sqrt{x_{2}-x_{1}} & \sqrt{x_{3}-x_{2}} & \cdots & \sqrt{x_{N}-x_{N-1}}
\end{bmatrix}.\label{eq:H31}
\end{equation}
Thus, 
\begin{equation}
\det(K_{F_{N}}^{\left(\mathbb{Z}\right)})=x_{1}\left(x_{2}-x_{1}\right)\cdots\left(x_{N}-x_{N-1}\right).
\end{equation}

In the setting of \secref{fac} (finite sums of standard Gaussians),
we have the following: Let $\left\{ x_{i}\right\} _{i=1}^{N}$ be
as in (\ref{eq:H31}), and let $1\leq n,m\leq N$. Let $\left\{ Z_{i}\right\} _{i=1}^{N}$
be a system i.i.d. standard Gaussians $N\left(0,1\right)$, i.e.,
independent identically distributed. Set 
\begin{equation}
V_{n}=Z_{1}\sqrt{x_{1}}+Z_{2}\sqrt{x_{2}-x_{1}}+\cdots+Z_{n}\sqrt{x_{n}-x_{n-1}}.
\end{equation}
Then one checks that 
\begin{equation}
\mathbb{E}\left(V_{n}V_{m}\right)=x_{n}\wedge x_{m}=K\left(x_{n},x_{m}\right)
\end{equation}
which is the desired Gaussian realization of $K$.

Alternatively, $K_{F_{N}}^{\left(\mathbb{Z}\right)}$ assumes the
following factorization via non-square matrices: Assume $F_{N}\subset\mathbb{Z}_{+}$,
then 
\begin{equation}
K_{F_{N}}^{\left(\mathbb{Z}\right)}=AA^{*},
\end{equation}
where $A$ is the $N\times x_{N}$ matrix such that 
\[
A_{n,k}=\begin{cases}
1 & \text{if \ensuremath{1\leq k\leq x_{n}}}\\
0 & \text{otherwise}
\end{cases}.
\]
That is, $A$ takes the form: 
\begin{equation}
A=\begin{bmatrix}\tikzmark{1L}1 & \cdots & 1\tikzmark{1R} & 0 & \cdots & \cdots & \cdots & \cdots & \cdots & 0\\
\\
\\
\tikzmark{2L}1 & \cdots & \cdots & \cdots & 1\tikzmark{2R} & 0 & \cdots & \cdots & \cdots & 0\\
\\
\\
\tikzmark{3L}1 & \cdots & \cdots & \cdots & \cdots & \cdots & 1\tikzmark{3R} & 0 & \cdots & 0\\
\\
\\
\vdots & \vdots &  &  & \vdots & \vdots &  & \vdots & \vdots & \vdots\\
\vdots & \vdots &  &  & \vdots & \vdots &  & \vdots & \vdots & 0\\
\tikzmark{NL}1 & 1 & \cdots & \cdots & \cdots & \cdots & \cdots & \cdots & 1 & 1\tikzmark{NR}
\end{bmatrix}.
\end{equation}

\tikz[overlay, remember picture, decoration={brace, amplitude=3pt}] {
  \draw[decorate,thick] (1R.south) -- (1L.south)        	node [midway,below=5pt] {$x_{1}$};
\draw[decorate,thick] (2R.south) -- (2L.south)        	node [midway,below=5pt] {$x_{2}$};
\draw[decorate,thick] (3R.south) -- (3L.south)        	node [midway,below=5pt] {$x_{3}$};
\draw[decorate,thick] (NR.south) -- (NL.south)        	node [midway,below=5pt] {$x_{N}$};
}
\end{rem}

\vspace{2em}

\begin{rem}[Spectrum of the matrices $K_{F}$; see also \cite{MR3034493}]
 It is known that the factorization as in (\ref{eq:H29-1}) can be
used to obtain the spectrum of positive definite matrices. The algorithm
is as follows: Let $K$ be a given p.d. matrix.

Initialization: $B:=K$;

Iterations: $k=1,2,\cdots,n-1$,
\begin{enumerate}
\item $B=AA^{*}$;
\item $B=A^{*}A$;
\end{enumerate}
Here $A$ in step (i) denotes the lower triangular matrix in the Cholesky
decomposition of $B$ (see (\ref{eq:H29-1})). Then $\lim_{n\rightarrow\infty}B$
converges to a diagonal matrix consisting of the eigenvalues of $K$.
\end{rem}

We now resume consideration of the general case of p.d. kernels $K$
on $X\times X$ and their restrictions: A setting for harmonic functions.
\begin{rem}
\label{rem:H10}In the general case of (\ref{eq:H2}) and \lemref{H1},
we still have a Laplace operator $\Delta$. It is a densely defined
symmetric operator on $\mathscr{H}\left(K\right)$. Moreover (general
case), 
\begin{equation}
\Delta_{\cdot}K\left(\cdot,x\right)=\delta_{x}\left(\cdot\right),\;\forall x\in X\label{eq:H23}
\end{equation}
(assuming that $\delta_{x}\in\mathscr{H}\left(K\right)$). The dot
``$\cdot$'' in (\ref{eq:H23}) refers to the action variable for
the operator $\Delta$. In other words, $K\left(\cdot,\cdot\right)$
is a generalized Greens kernel.
\end{rem}

\begin{defn}
Let $X\times X\xrightarrow{\;K\:}\mathbb{C}$ be given p.d., and assume
\begin{equation}
\left\{ \delta_{x}\right\} _{x\in X}\subset\mathscr{H}\left(K\right).\label{eq:H24}
\end{equation}
Let $\Delta$ denote the induced Laplace operator. A function $h$
(in $\mathscr{H}\left(K\right)$) is said to be \emph{harmonic} iff
(Def.) $\Delta h=0$.
\end{defn}

\begin{cor}
Let $\left(X,K,\mathscr{H}\left(K\right)\right)$ be as above. Assume
(\ref{eq:H24}), and let $\Delta$ be the induced Laplace operator.
Then we have the following orthogonal decomposition for $\mathscr{H}\left(K\right)$:
\begin{equation}
\mathscr{H}\left(K\right)=\left\{ h\mathrel{;}\Delta h=0\right\} \oplus clospan^{\mathscr{H}\left(K\right)}\left(\left\{ \delta_{x}\right\} _{x\in X}\right)\label{eq:H25}
\end{equation}
where ``clospan'' in (\ref{eq:H25}) refers to the norm in $\mathscr{H}\left(K\right)$.
\end{cor}

\begin{proof}
It is immediate from (\ref{eq:H1}) that 
\begin{equation}
\left\{ h\in\mathscr{H}\left(K\right)\mathrel{;}\Delta h=0\right\} =\left(\left\{ \delta_{x}\right\} _{x\in X}\right)^{\perp}\label{eq:H26}
\end{equation}
where the orthogonality ``$\perp$'' in (\ref{eq:H26}) refers to
the inner product $\left\langle \cdot,\cdot\right\rangle _{\mathscr{H}\left(K\right)}$.
Since, by Hilbert space geometry, $\left(\left\{ \delta_{x}\right\} _{x\in X}\right)^{\perp\perp}=clospan^{\mathscr{H}\left(K\right)}\left(\left\{ \delta_{x}\right\} _{x\in X}\right)$,
we only need to observe that $\left\{ h\in\mathscr{H}\left(K\right)\mathrel{;}\Delta h=0\right\} $
is closed in $\mathscr{H}\left(K\right)$. But this is immediate from
(\ref{eq:H1}).
\end{proof}
\begin{cor}[Duality]
\label{cor:H14} Let $X\times X\xrightarrow{\;K\;}\mathbb{R}$ be
given, assumed p.d., and let $S\subset X$ be a countable subset such
that 
\begin{equation}
\mathscr{D}\left(S\right):=\left\{ \delta_{x}\right\} _{x\in S}\subset\mathscr{H}(K^{\left(S\right)}).
\end{equation}
\begin{enumerate}
\item Then the following duality holds for the two induced kernels:
\begin{align}
K^{\left(S\right)} & :=K\big|_{S\times S},\;\text{and}\\
D^{\left(S\right)}\left(x,y\right) & :=\left\langle \delta_{x},\delta_{y}\right\rangle _{\mathscr{H}\left(K^{\left(S\right)}\right)},\;\forall\left(x,y\right)\in S\times S;
\end{align}
both p.d. kernels on $S\times S$.

For every pair $x,y\in S$, we have the following matrix-inversion
formula:
\begin{equation}
\sum_{z\in S}D^{\left(S\right)}\left(x,z\right)K^{\left(S\right)}\left(z,y\right)=\delta_{x,y},\label{eq:H32}
\end{equation}
where the summation on the LHS in (\ref{eq:H32}) is a limit over
a net of finite subsets $\left\{ F_{i}\right\} _{i\in\mathbb{N}}$,
$F_{1}\subset F_{2}\subset\cdots$, s.t. $\cup_{i}F_{i}=S$; and the
result is independent of choice of net.
\item \label{enu:corH14-2}We get an \uline{induced graph} with $S$
as the set of vertices, and edge set $E$ as follows: $E\subset\left(S\times S\right)\backslash\left(\text{diagonal}\right)$.

An edge is a pair $\left(x,y\right)\in\left(S\times S\right)\backslash\left(\text{diagonal}\right)$
such that 
\[
\left\langle \delta_{x},\delta_{y}\right\rangle _{\mathscr{H}\left(K^{\left(S\right)}\right)}\neq0.
\]

\end{enumerate}
\end{cor}

\begin{proof}
The result follows from an application of Corollaries \ref{cor:H7}
and \ref{cor:H8}, and \remref{H10}.
\end{proof}
Let $X$, $K$, and $S$ be as stated, $S$ countable infinite, with
assumptions as in the previous two results. We showed that then the
subset $S$ acquires the structure of a vertex set in an induced infinite
graph (\corref{H14} (\ref{enu:corH14-2})). If $\Delta$ denotes
the corresponding graph Laplacian, then the following boundary value
problem is of great interest: Make precise the boundary conditions
at \textquotedblleft infinity\textquotedblright{} for this graph Laplacian
$\Delta$. An answer to this will require identification of Hilbert
space, and limit at \textquotedblleft infinity.\textquotedblright{}
The result below is such an answer, and the limit notion will be,
limit over the filter of all finite subsets in $S$; see \corref{H7}.
Another key tool in the arguments below will again be the net of orthogonal
projections $\left\{ P_{F}\right\} $ from \lemref{H4}, and the convergence
results from Corollaries \ref{cor:H5} and \ref{cor:H7}.
\begin{cor}
Let $X\times X\xrightarrow{\;K\;}\mathbb{R}$, and $S\subset X$ be
as in the statement of \corref{H14}. Let $\mathscr{F}_{fin}\left(S\right)$
denote the filter of finite subsets $F\subset S$. Let $\Delta=\Delta_{S}$
be the graph Laplacian defined in (\ref{eq:H2}), i.e., 
\[
\left(\Delta h\right)\left(x\right):=\left\langle \delta_{x},h\right\rangle _{\mathscr{H}(K^{\left(S\right)})},
\]
for all $x\in S$, $h\in\mathscr{H}(K^{\left(S\right)})$. Then the
following equivalent conditions hold:
\begin{enumerate}
\item For all $h\in\mathscr{H}(K^{\left(S\right)})$, 
\begin{align}
\left\Vert h\right\Vert _{\mathscr{H}(K^{\left(S\right)})}^{2} & =\sup_{F\in\mathscr{F}_{fin}\left(S\right)}\left\langle h\big|_{F},\Delta P_{F}h\right\rangle _{l^{2}\left(F\right)}\\
 & =\sup_{F\in\mathscr{F}_{fin}\left(S\right)}\left\langle h\big|_{F},K_{F}^{-1}\left(h\big|_{F}\right)\right\rangle _{l^{2}\left(F\right)}.\nonumber 
\end{align}
\item For $\forall F\in\mathscr{F}_{fin}\left(S\right)$, $x\in F$, $h\in\mathscr{H}(K^{\left(S\right)})$,
\begin{equation}
\left(\Delta\left(P_{F}h\right)\right)\left(x\right)=\left(K_{F}^{-1}\left(h\big|_{F}\right)\right)\left(x\right).\label{eq:H46}
\end{equation}
\item $K_{F}\Delta P_{F}h=h\big|_{F}$.
\end{enumerate}
\end{cor}

\begin{proof}
On account of \corref{H8}, we only need to verify (\ref{eq:H46}).
Let $F\in\mathscr{F}_{fin}\left(S\right)$, $h\in\mathscr{H}(K^{\left(S\right)})$,
then we proved that 
\begin{align}
\left(P_{F}h\right)\left(\cdot\right) & =\sum_{y\in F}\xi_{y}K\left(\cdot,y\right)\;\text{with}\label{eq:H47}\\
\xi_{y} & =\left(K_{F}^{-1}\left(h\big|_{F}\right)\right)\left(y\right).
\end{align}
Now apply $\left\langle \delta_{x},\cdot\right\rangle _{\mathscr{H}(K^{\left(S\right)})}$
to both sides in (\ref{eq:H47}); and we get
\begin{equation}
\left(\Delta\left(P_{F}h\right)\right)\left(x\right)=\xi_{x}\label{eq:H49}
\end{equation}
where we used $\left\langle \delta_{x},K\left(\cdot,y\right)\right\rangle _{\mathscr{H}(K^{\left(S\right)})}=\delta_{x,y}$.
The desired conclusion (\ref{eq:H46}) now follows from (\ref{eq:H49}).
Also note that $\left(\Delta\left(P_{F}h\right)\right)\left(x\right)=0$
if $x\in X\backslash F$.
\end{proof}

\subsection{Canonical isometries computed from point processes}

Below we consider p.d. kernels $K$ defined initially on $X\times X$.
Our present aim is to consider restrictions to $S\times S$ when $S$
is a suitable subset of $X$. Our first observation is the identification
of a canonical isometry $T_{S}$ between the respective reproducing
kernel Hilbert spaces; $T_{S}$ identifying $\mathscr{H}(K^{\left(S\right)})$
as an isometric subspace inside $\mathscr{H}(K)$. This isometry $T_{S}$
exists in general. However, we shall show that, when the subset $S$
is further restricted, the respective RKHSs, and isometry $T_{S}$
will admit explicit characterizations. For example, if $S$ is countable,
and is the Dirac functions $\delta_{s}$, $s\in S$, are in $\mathscr{H}(K^{\left(S\right)})$
we shall show that this setting leads to a point process. In this
case, we further identify an induced (infinite) graph with the set
$S$ as vertices, and with associated edges defined by an induced
$\delta_{s}$ kernel.
\begin{thm}
\label{thm:H14}Let $X\times X\xrightarrow{\;K\;}\mathbb{C}$ be a
p.d. kernel, and let $S\subset X$ be a subset. Set $K^{\left(S\right)}:=K\big|_{S\times S}$.
Let $\mathscr{H}\left(K\right)$, and $\mathscr{H}(K^{\left(S\right)})$,
be the respective RKHSs.
\begin{enumerate}
\item Then there is a canonical isometric embedding 
\[
\mathscr{H}(K^{\left(S\right)})\xrightarrow{\;T\;}\mathscr{H}\left(K\right),
\]
given by the following formula: For $s\in S$, set 
\begin{equation}
T(K^{\left(S\right)}\left(\cdot,s\right))=K\left(\cdot,s\right).\label{eq:H35}
\end{equation}
(Note that $K^{\left(S\right)}\left(\cdot,s\right)$ on the LHS in
(\ref{eq:H35}) is a function on $S$, while $K\left(\cdot,s\right)$
on the RHS is a function on $X$.)
\item \label{enu:H14-2}The adjoint operator $T^{*}$,
\begin{equation}
\mathscr{H}\left(K\right)\xrightarrow{\;T^{*}\;}\mathscr{H}(K^{\left(S\right)})
\end{equation}
is given by restriction, i.e., if $f\in\mathscr{H}(K)$, and $s\in S$,
then $\left(T^{*}f\right)\left(s\right)=f\left(s\right)$; or equivalently,
for all $f\in\mathscr{H}(K)$, 
\begin{equation}
T^{*}f=f\big|_{S}.\label{eq:H37}
\end{equation}
\end{enumerate}
\end{thm}

\begin{proof}
To show that $T$ in (\ref{eq:H35}) is isometric, proceed as follows:
Let $\left\{ s_{i}\right\} _{i=1}^{N}$ be a finite subset of $S$,
and $\left\{ \xi_{i}\right\} _{i=1}^{N}\in\mathbb{C}^{N}$, then 
\begin{eqnarray*}
\left\Vert T(\sum\nolimits _{i}\xi_{i}K^{\left(S\right)}\left(\cdot,s_{i}\right))\right\Vert _{\mathscr{H}\left(K\right)}^{2} & = & \left\Vert \sum\nolimits _{i}\xi_{i}T(K^{\left(S\right)}\left(\cdot,s_{i}\right))\right\Vert _{\mathscr{H}\left(K\right)}^{2}\\
 & \underset{\text{by \ensuremath{\left(\ref{eq:H35}\right)}}}{=} & \left\Vert \sum\nolimits _{i}\xi_{i}K\left(\cdot,s_{i}\right)\right\Vert _{\mathscr{H}\left(K\right)}^{2}\\
 & \underset{\text{by \ensuremath{\left(\ref{eq:A3}\right)}}}{=} & \sum\nolimits _{i}\sum\nolimits _{j}\overline{\xi}_{i}\xi_{j}K\left(s_{i},s_{j}\right)\\
 & = & \sum\nolimits _{i}\sum\nolimits _{j}\overline{\xi}_{i}\xi_{j}K^{\left(S\right)}\left(s_{i},s_{j}\right)\\
 & = & \left\Vert \sum\nolimits _{i}\xi_{i}K^{\left(S\right)}\left(\cdot,s_{i}\right)\right\Vert _{\mathscr{H}\left(K^{\left(S\right)}\right)}^{2}
\end{eqnarray*}
which is the desired isometric property.

We now turn to (\ref{eq:H37}), the restriction formula: Let $s\in S$,
and $f\in\mathscr{H}\left(K\right)$, then 
\begin{eqnarray}
\left\langle T(K^{\left(S\right)}\left(\cdot,s\right)),f\right\rangle _{\mathscr{H}\left(K\right)} & = & \left\langle K^{\left(S\right)}\left(\cdot,s\right),T^{*}f\right\rangle _{\mathscr{H}\left(K^{\left(S\right)}\right)}\label{eq:H38}\\
 & \underset{\text{by \ensuremath{\left(\ref{eq:A4}\right)}}}{=} & \left(T^{*}f\right)\left(s\right).\nonumber 
\end{eqnarray}
But, for the LHS in (\ref{eq:H38}), we have 
\[
\left\langle T(K^{\left(S\right)}\left(\cdot,s\right)),f\right\rangle _{\mathscr{H}\left(K\right)}\underset{\text{by \ensuremath{\left(\ref{eq:H35}\right)}}}{=}\left\langle K\left(\cdot,s\right),f\right\rangle _{\mathscr{H}\left(K\right)}\underset{\text{by \ensuremath{\left(\ref{eq:A4}\right)}}}{=}f\left(s\right);
\]
and so the desired formula (\ref{eq:H37}) follows.
\end{proof}
\begin{rem}
\label{rem:H8}\textbf{The canonical isometry for \exaref{H8} ($\mathbb{Z}$-discretization
of the covariance function for Brownian motion on $\mathbb{R}$).}
From \thmref{H14}, we know that the canonical isometry $T$ maps
$\mathscr{H}(K^{\left(Z\right)})$ into $\mathscr{H}\left(K\right)$;
see (\ref{eq:H22}). But (\ref{eq:H19}) and (\ref{eq:H24-1}) in
the Example offer exact characterization of these two Hilbert spaces.
So, in the special case of \exaref{H8}, the canonical isometry $T$
maps from functions $\Phi$ on $\mathbb{Z}$ into functions on $\mathbb{R}$.
In view of (\ref{eq:H19}), this assignment turns out to be a precise
spline realization of the point grids realized by these sequences
$\Phi$.

Below we present an explicit formula, and graphics, for the spline
realizations. By (\ref{eq:H26-1}), the embedding of $\delta_{n}$
from $\mathscr{H}(K^{\left(\mathbb{Z}\right)})$ into $\mathscr{H}\left(K\right)$
is given by 
\[
\left(T\delta_{n}\right)\left(x\right)=2K\left(x,n\right)-K\left(x,n+1\right)-K\left(x,n-1\right),\;\forall x\in\mathbb{R}.
\]
See \figref{H1}. Therefore, for all $h\in\mathscr{H}\left(K\right)$,
we get 
\begin{align*}
\left(T^{*}h\right)\left(m\right) & =\sum_{n\in\mathbb{Z}}h\left(n\right)\delta_{n}\left(m\right),\;m\in\mathbb{Z},\;\text{and}\\
\left(TT^{*}h\right)\left(x\right) & =\sum_{n\in\mathbb{Z}}h\left(n\right)\left(2K\left(x,n\right)-K\left(x,n+1\right)-K\left(x,n-1\right)\right),\;x\in\mathbb{R}
\end{align*}
which is the spline interpolation.
\end{rem}

\begin{figure}
\includegraphics[width=0.6\columnwidth]{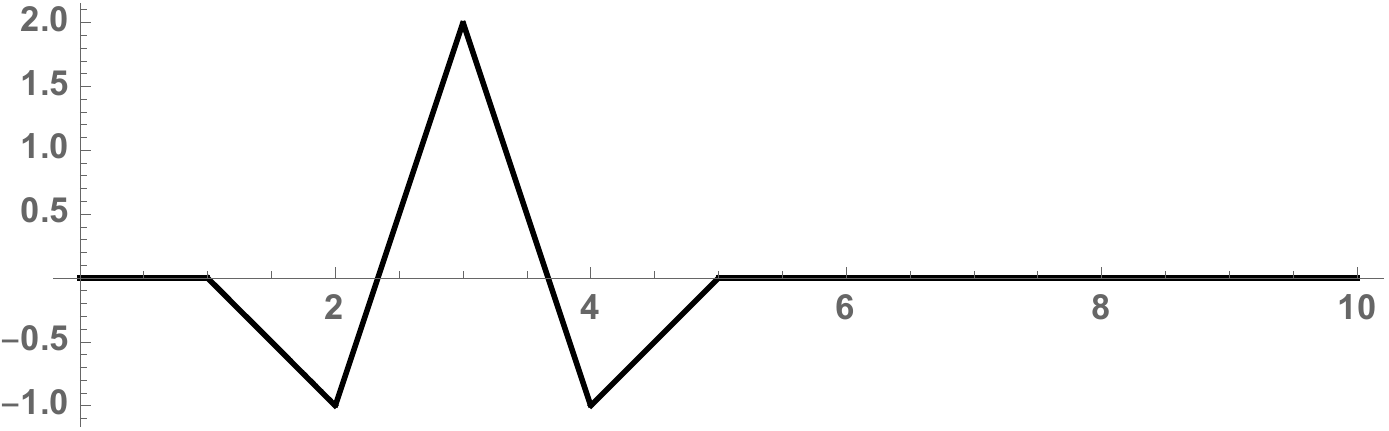}

\caption{\label{fig:H1}Isometric extrapolation from functions on $\mathbb{Z}$
to functions on $\mathbb{R}$. An illustration of the isometric embedding
of $\delta_{n}$ from $\mathscr{H}(K^{\left(\mathbb{Z}\right)})$
into $\mathscr{H}\left(K\right)$, with $n=3$.}
\end{figure}

\begin{cor}
Let $X\times X\xrightarrow{\;K\;}\mathbb{C}$ be a p.d. kernel, and
let $S\subset X$ be a subset. Assume further that $\left\{ \delta_{s}\right\} _{s\in S}\subset\mathscr{H}(K^{\left(S\right)})$.
Then every finitely supported function $h$ on $S$ is in $\mathscr{H}(K^{\left(S\right)})$,
and we have the following generalized spline interpolation; i.e.,
isometrically extending $h$ from $S$ to $X$:
\begin{equation}
\widetilde{h}\left(x\right)=\sup\nolimits _{F\supset F_{0}}\sum\nolimits _{y\in F}\left(K_{F}^{-1}h_{F}\right)\left(y\right)K\left(y,x\right),\;x\in X,
\end{equation}
where $F_{0}=suppt\left(h\right)$, and the sup is taken over the
filter of all finite subsets of $X$ containing $F_{0}$.
\end{cor}

\begin{proof}
Assume $h\in\mathscr{H}(K^{\left(S\right)})$, supported on a finite
subset $F_{0}\subset S$. Then, 
\begin{align*}
\widetilde{h}\left(x\right):=Th\left(x\right) & =T\left(\sum\nolimits _{s\in F_{0}}h\left(s\right)\delta_{s}\right)\left(x\right)\\
 & =\sum\nolimits _{s\in F_{0}}h\left(s\right)\left(T\delta_{s}\right)\left(x\right)\\
 & =\sum\nolimits _{s\in F_{0}}h\left(s\right)\sup\nolimits _{F\supset F_{0}}\left(P_{F}\delta_{s}\right)\left(x\right)\\
 & =\sup\nolimits _{F\supset F_{0}}P_{F}\left(\sum\nolimits _{s\in F_{0}}h\left(s\right)\delta_{s}\right)\left(x\right)\\
 & =\sup\nolimits _{F\supset F_{0}}\left(P_{F}h_{F_{0}}\right)\left(x\right)\\
 & =\sup\nolimits _{F\supset F_{0}}\sum\nolimits _{y\in F}\left(K_{F}^{-1}h_{F}\right)\left(y\right)K\left(x,y\right),
\end{align*}
where the last step follows from (\ref{eq:H8}), and $P_{F}$ is the
orthogonal projection from $\mathscr{H}\left(K\right)$ onto the subspace
$\mathscr{H}_{K}\left(F\right)$.
\end{proof}
\begin{cor}
\label{cor:H15}Let $X\times X\xrightarrow{\;K\;}\mathbb{C}$, p.d..
be given, and let $S\subset X$ be a subset. Let $T=T_{S}$, $\mathscr{H}(K^{\left(S\right)})\xrightarrow{\;T\;}\mathscr{H}\left(K\right)$,
be the canonical isometry. Then a function $f$ in $\mathscr{H}\left(K\right)$
satisfies $\left\langle f,T(\mathscr{H}(K^{\left(S\right)}))\right\rangle _{\mathscr{H}\left(K\right)}=0$
if and only if 
\begin{equation}
f\left(s\right)=0\;\text{for all \ensuremath{s\in S}.}\label{eq:H39}
\end{equation}
\end{cor}

\begin{proof}
Immediate from part (\ref{enu:H14-2}) in \thmref{H14}.
\end{proof}
\begin{rem}
Let $\left(X,K,S\right)$ be as in \corref{H15}, and let $T_{S}$
be the canonical isometry. Let $P_{S}:=T_{S}T_{S}^{*}$ be the corresponding
projection. Then $I_{\mathscr{H}\left(K\right)}-P_{S}$ is the projection
onto the subspace given in (\ref{eq:H39}).
\end{rem}

\begin{cor}
Let $X\times X\xrightarrow{\;K\;}\mathbb{C}$ be given p.d.; and let
$S\subset X$ be a subset with induced kernel 
\begin{equation}
K^{\left(S\right)}:=K\big|_{S\times S}.\label{eq:H40}
\end{equation}
Consider the two sets $\mathscr{F}\left(S\right)$ and $\mathscr{F}(K^{\left(S\right)})$
from (\ref{eq:D2}) and \thmref{E1}. Let $T_{S}:\mathscr{H}(K^{\left(S\right)})\rightarrow\mathscr{H}\left(K\right)$
be the canonical isometry (\ref{eq:H35}) in \thmref{H14}. Then the
following implication holds:
\begin{eqnarray}
\left(\left\{ k_{x}\right\} _{x\in X},\mu\right) & \in & \mathscr{F}\left(K\right)\label{eq:H41}\\
 & \Downarrow\nonumber \\
\left(\left\{ k_{s}\right\} _{s\in S},\mu\right) & \in & \mathscr{F}(K^{\left(S\right)})\label{eq:H42}
\end{eqnarray}
\end{cor}

\begin{proof}
Assuming (\ref{eq:H41}), we get the representation (\ref{eq:D2}):
\begin{equation}
K\left(x,y\right)=\int_{M}\overline{k}_{x}k_{y}d\mu,\;\forall\left(x,y\right)\in X\times X.
\end{equation}
But then, for all $\left(s_{1},s_{2}\right)\in S\times S$, we then
have 
\begin{eqnarray*}
K^{\left(S\right)}\left(s_{1},s_{2}\right) & = & \left\langle T_{S}(K^{\left(S\right)}\left(\cdot,s_{1}\right)),T_{S}(K^{\left(S\right)}\left(\cdot,s_{2}\right))\right\rangle _{\mathscr{H}\left(K\right)}\\
 & \underset{\text{by \ensuremath{\left(\ref{eq:H40}\right)}}}{=} & K\left(s_{1},s_{2}\right)\\
 & = & \int_{M}\overline{k}_{s_{1}}k_{s_{2}}d\mu,
\end{eqnarray*}
which is the desired conclusion.
\end{proof}

\section{Boundary value problems}

Our setting in the present section is the discrete case, i.e., RKHSs
of functions defined on a prescribed countable infinite discrete set
$S$. We are concerned with a characterization of those RKHSs $\mathscr{H}$
which contain the Dirac masses $\delta_{x}$ for all points $x\in S$.
Of the examples and applications where this question plays an important
role, we emphasize two: (i) discrete Brownian motion-Hilbert spaces,
i.e., discrete versions of the Cameron-Martin Hilbert space; (ii)
energy-Hilbert spaces corresponding to graph-Laplacians.

The problems addressed here are motivated in part by applications
to analysis on infinite weighted graphs, to stochastic processes,
and to numerical analysis (discrete approximations), and to applications
of RKHSs to machine learning. Readers are referred to the following
papers, and the references cited there, for details regarding this:
\cite{MR3231624,MR2966130,MR2793121,MR3286496,MR3246982,MR2862151,MR3096457,MR3049934,MR2579912,MR741527,MR3024465}.

The discrete case can be understood as restrictions of analogous PDE-models.
In traditional numerical analysis, one builds discrete and algorithmic
models (finite element methods), each aiming at finding approximate
solutions to PDE-boundary value problems. They typically use multiresolution-subdivision
schemes, applied to the continuous domain, subdividing into simpler
discretized parts, called finite elements. And with variational methods,
one then minimize various error-functions. In this paper, we turn
the tables: our object of study are the discrete models, and analysis
of suitable continuous PDE boundary problems serve as a tool for solutions
in the discrete world.
\begin{defn}
\label{def:dmp}Let $X\times X\xrightarrow{\;K\;}\mathbb{C}$ be a
given p.d. kernel on $X$. The RKHS $\mathscr{H}=\mathscr{H}\left(K\right)$
is said to have the \emph{discrete mass} property ($\mathscr{H}$
is called a \emph{discrete RKHS}), if $\delta_{x}\in\mathscr{H}$,
for all $x\in X$.
\end{defn}

In fact, it is known (\cite{MR3507188}) that every fundamental solution
for a Dirichlet boundary value problem on a bounded open domain $\Omega$
in $\mathbb{R}^{\nu}$, allows for discrete restrictions (i.e., vertices
sampled in $\Omega$), which have the desired ``discrete mass''
property.

We recall the following result to stress the distinction of the discrete
models vs their continuous counterparts.

Let $\Omega$ be a bounded, open, and connected domain in $\mathbb{R}^{\nu}$
with smooth boundary $\partial\Omega$. Let $K:\Omega\times\Omega\rightarrow\mathbb{R}$
continuous, p.d., given as the Green's function of $\Delta_{0}$,
where
\begin{equation}
\begin{split} & \Delta_{0}:=-\sum_{j=1}^{\nu}\left(\frac{\partial}{\partial x_{j}}\right)^{2},\\
 & dom\left(\Delta_{0}\right)=\left\{ f\in L^{2}\left(\Omega\right)\:\big|\:\Delta f\in L^{2}\left(\Omega\right),\;\mbox{and }f\big|_{\partial\Omega}\equiv0\right\} .
\end{split}
\label{eq:e4}
\end{equation}
for the Dirichlet boundary condition. Thus, $\Delta_{0}$ is positive
selfadjoint, and 
\begin{align}
 & \Delta_{0}K=\delta\left(x-y\right)\text{ on \ensuremath{\Omega\times\Omega}}\label{eq:m1}\\
 & K\left(x,\cdot\right)\big|_{\partial\Omega}\equiv0.\label{eq:m2}
\end{align}

Let $\mathscr{H}_{CM}\left(\Omega\right)$ be the corresponding Cameron-Martin
RKHS.

For $\nu=1$, $\Omega=\left(0,1\right)$, take
\begin{equation}
\begin{split}\mathscr{H}_{CM}\left(0,1\right)= & \Big\{ f\:\big|\:f'\in L^{2}\left(0,1\right),\;f\left(0\right)=f\left(1\right)=0,\\
 & \left\Vert f\right\Vert _{CM}^{2}:=\int_{0}^{1}\left|f'\right|^{2}dx<\infty\Big\}
\end{split}
\label{eq:m3}
\end{equation}

For $\nu>1$, let
\begin{equation}
\begin{split}\mathscr{H}_{CM}\left(\Omega\right)= & \left\{ f\:\big|\:\nabla f\in L^{2}\left(\Omega\right),\:f\big|_{\partial\Omega}\equiv0,\:\left\Vert f\right\Vert _{CM}^{2}:=\int_{\Omega}\left|\nabla f\right|^{2}dx<\infty\right\} ,\\
 & \text{ where }\nabla=\left(\frac{\partial}{\partial x_{1}},\frac{\partial}{\partial x_{2}},\cdots,\frac{\partial}{\partial x_{\nu}}\right).
\end{split}
\label{eq:m4}
\end{equation}

\begin{thm}
\label{thm:main}Let $\Omega$, and $S\subset\Omega$, be given. Then

\begin{enumerate}
\item Discrete case: Fix $S\subset\Omega$, $\#S=\aleph_{0}$, where $S=\left\{ x_{j}\right\} _{j=1}^{\infty}$,
$x_{j}\in\Omega$. Assume $\exists\varepsilon>0$ s.t. $\left\Vert x_{i}-x_{j}\right\Vert \geq\varepsilon$,
$\forall i,j$, $i\neq j$. Let 
\[
\mathscr{H}\left(S\right)=\text{RKHS of \ensuremath{K^{\left(S\right)}:=K\big|_{S\times S}}};
\]
then $\delta_{x_{j}}\in\mathscr{H}\left(S\right)$. 
\item Continuous case; by contrast: $K_{x}^{\left(S\right)}\in\mathscr{H}_{CM}\left(S\right)$,
but $\delta_{x}\notin\mathscr{H}_{CM}\left(\Omega\right)$, $x\in\Omega$.
\end{enumerate}
\end{thm}

\begin{proof}
The result follows from an application of Corollaries \ref{cor:H7}
and \ref{cor:H8}. It extends earlier results \cite{MR3450534,MR3507188}
by the co-authors.
\end{proof}

\section{Sampling in $\mathscr{H}\left(K\right)$}

In the present section, we study classes of reproducing kernels $K$
on general domains with the property that there are non-trivial restrictions
to countable discrete sample subsets $S$ such that every function
in $\mathscr{H}\left(K\right)$ has an $S$-sample representation.
In this general framework, we study properties of positive definite
kernels $K$ with respect to sampling from ``small\textquotedblright{}
subsets, and applying to all functions in the associated Hilbert space
$\mathscr{H}\left(K\right)$.

We are motivated by concrete kernels which are used in a number of
applications, for example, on one extreme, the Shannon kernel for
band-limited functions, which admits many sampling realizations; and
on the other, the covariance kernel of Brownian motion which has no
non-trivial countable discrete sample subsets.
\begin{defn}
\label{def:J1}Let $X\times X\xrightarrow{\;K\;}\mathbb{C}$ be a
p.d. kernel, and $\mathscr{H}\left(K\right)$ be the associated RKHS.
We say that $K$ has non-trivial sampling property, if there exists
a countable subset $S\subset X$, and $a,b\in\mathbb{R}_{+}$, such
that 
\begin{equation}
a\sum_{s\in S}\left|f\left(s\right)\right|^{2}\leq\left\Vert f\right\Vert _{\mathscr{H}\left(K\right)}^{2}\leq b\sum_{s\in S}\left|f\left(s\right)\right|^{2},\quad\forall f\in\mathscr{H}\left(K\right).\label{eq:sp1}
\end{equation}
If equality holds in (\ref{eq:sp1}) with $a=b=1$, then we say that
$\left\{ K\left(\cdot,s\right)\right\} _{s\in S}$ is a Parseval frame.
(Also see \defref{G1}.)
\end{defn}

It follows that sampling holds in the form
\[
f\left(x\right)=\sum_{s\in S}f\left(s\right)K\left(x,s\right),\quad\forall f\in\mathscr{H}\left(K\right),\:\forall x\in X
\]
if and only if $\left\{ K\left(\cdot,s\right)\right\} _{s\in S}$
is a Parseval frame.
\begin{lem}
\label{lem:fr}Suppose $K$, $X$, $a$, $b$, and $S$ satisfy the
condition in (\ref{eq:sp1}), then the linear span of $\left\{ K\left(\cdot,s\right)\right\} _{s\in S}$
is dense in $\mathscr{H}\left(K\right)$. Moreover, there is a positive
operator $B$ in $\mathscr{H}\left(K\right)$ with bounded inverse
such that 
\[
f\left(\cdot\right)=\sum_{s\in S}\left(Bf\right)\left(s\right)K\left(\cdot,s\right)
\]
is a convergent interpolation formula valid for all $f\in\mathscr{H}\left(K\right)$.

Equivalently, 
\[
f\left(x\right)=\sum_{s\in S}f\left(s\right)B\left(K\left(\cdot,s\right)\right)\left(x\right),\;\text{for all \ensuremath{x\in X}.}
\]
\end{lem}

\begin{proof}
Define $A:\mathscr{H}\left(K\right)\rightarrow l^{2}\left(S\right)$
by $\left(Af\right)\left(s\right)=f\left(s\right)$, $s\in S$. Then
the adjoint operator $A^{*}:l^{2}\left(S\right)\rightarrow\mathscr{H}\left(K\right)$
is given by $A^{*}\xi=\sum_{s\in S}\xi_{s}K\left(\cdot,s\right)$,
$\forall\xi\in l^{2}\left(S\right)$, and 
\[
A^{*}Af=\sum_{s\in S}f\left(s\right)K\left(\cdot,s\right)
\]
holds in $\mathscr{H}\left(K\right)$, with $\mathscr{H}\left(K\right)$-norm
convergence. Now set $B=\left(A^{*}A\right)^{-1}$, and note that
$\left\Vert B\right\Vert _{\mathscr{H}\left(K\right)\rightarrow\mathscr{H}\left(K\right)}\leq a^{-1}$,
where $a$ is in the lower bound in (\ref{eq:sp1}).
\end{proof}
\begin{thm}
\label{thm:ps}Let $K:X\times X\rightarrow\mathbb{R}$ be a p.d. kernel,
and let $S\subset X$ be a countable discrete subset. For all $s\in S$,
set $K_{s}\left(\cdot\right)=K\left(\cdot,s\right)$. Then TFAE:
\begin{enumerate}
\item \label{enu:ps1}The family $\left\{ K_{s}\right\} _{s\in S}$ is a
Parseval frame in $\mathscr{H}\left(K\right)$; 
\item \label{enu:ps2}
\[
\left\Vert f\right\Vert _{\mathscr{H}\left(K\right)}^{2}=\sum_{s\in S}\left|f\left(s\right)\right|^{2},\;\forall f\in\mathscr{H}\left(K\right);
\]
\item \label{enu:ps3}
\[
K\left(x,x\right)=\sum_{s\in S}\left|K\left(x,s\right)\right|^{2},\;\forall x\in X;
\]
\item \label{enu:ps4}
\[
f\left(x\right)=\sum_{s\in S}f\left(s\right)K\left(x,s\right),\;\forall f\in\mathscr{H}\left(K\right),\:\forall x\in X,
\]
where the sum converges in the norm of $\mathscr{H}\left(K\right)$.
\end{enumerate}
\end{thm}

\begin{proof}
The proof is simple, and follows the steps in the proof of \lemref{G2}.
Details are left to the reader.
\end{proof}
We now turn to dichotomy: Existence of countably discrete sampling
sets vs non-existence.
\begin{example}
\label{exa:shan}Let $X=\mathbb{R}$, and let $K:\mathbb{R}\times\mathbb{R}\rightarrow\mathbb{R}$
be the Shannon kernel, where 
\begin{align}
K\left(x,y\right) & :=\text{sinc}\,\pi\left(x-y\right)\nonumber \\
 & =\frac{\sin\pi\left(x-y\right)}{\pi\left(x-y\right)},\quad\forall x,y\in\mathbb{R}.\label{eq:sp5}
\end{align}

We may choose $S=\mathbb{Z}$, and then $\left\{ K\left(\cdot,n\right)\right\} _{n\in\mathbb{Z}}$
is even an orthonormal basis (ONB) in $\mathscr{H}\left(K\right)$,
but there are many other examples of countable discrete subsets $S\subset\mathbb{R}$
such that (\ref{eq:sp1}) holds for finite $a,b\in\mathbb{R}_{+}$. 

The RKHS $\mathscr{H}\left(K\right)$ in (\ref{eq:sp5}) is the Hilbert
space $\subset L^{2}\left(\mathbb{R}\right)$ consisting of all $f\in L^{2}\left(\mathbb{R}\right)$
such that $suppt(\hat{f})\subset\left[-\pi,\pi\right]$, where ``suppt''
stands for support of the Fourier transform $\hat{f}$. Note $\mathscr{H}\left(K\right)$
consists of functions on $\mathbb{R}$ which have entire analytic
extensions to $\mathbb{C}$. Using the above observations, we get
\begin{align*}
f\left(x\right) & =\sum_{n\in\mathbb{Z}}f\left(n\right)K\left(x,n\right)\\
 & =\sum_{n\in\mathbb{Z}}f\left(n\right)\text{sinc}\,\pi\left(x-n\right),\quad\forall x\in\mathbb{R},\:\forall f\in\mathscr{H}\left(K\right).
\end{align*}
\end{example}

\begin{example}
Let $K$ be the covariant kernel of standard Brownian motion, with
$X:=[0,\infty)$ or $[0,1)$, and 
\begin{equation}
K\left(x,y\right):=x\wedge y=\min\left(x,y\right),\;\forall\left(x,y\right)\in X\times X.\label{eq:sp6}
\end{equation}
\end{example}

\begin{thm}
\label{thm:bm}Let $K$, $X$ be as in (\ref{eq:sp6}); then there
is no countable discrete subset $S\subset X$ such that $\left\{ K\left(\cdot,s\right)\right\} _{s\in S}$
is dense in $\mathscr{H}\left(K\right)$. 
\end{thm}

\begin{proof}
Suppose $S=\left\{ x_{n}\right\} $, where 
\begin{equation}
0<x_{1}<x_{2}<\cdots<x_{n}<x_{n+1}<\cdots;\label{eq:sp7}
\end{equation}
then consider the following function
\begin{equation}
\raisebox{-6mm}{\includegraphics[width=0.7\textwidth]{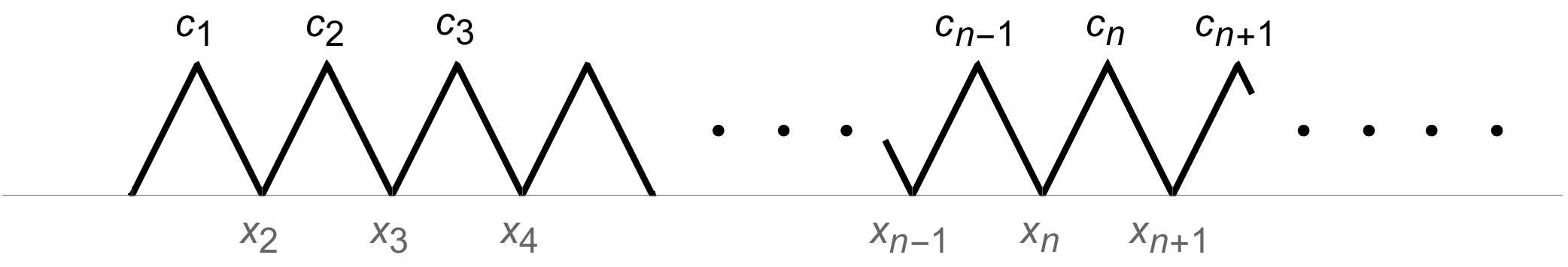}}\label{eq:sp9}
\end{equation}

On the respective intervals $\left[x_{n},x_{n+1}\right]$, the function
$f$ is as follows:
\[
f\left(x\right)=\begin{cases}
c_{n}\left(x-x_{n}\right) & \text{if }x_{n}\leq x\leq\frac{x_{n}+x_{n+1}}{2}\\
c_{n}\left(x_{n+1}-x\right) & \text{if }\frac{x_{n}+x_{n+1}}{2}<x\leq x_{n+1}.
\end{cases}
\]
In particular, $f\left(x_{n}\right)=f\left(x_{n+1}\right)=0$, and
on the midpoints: 
\[
f\left(\frac{x_{n}+x_{n+1}}{2}\right)=c_{n}\frac{x_{n+1}-x_{n}}{2},
\]
see \figref{stooth}.
\begin{figure}[H]
\includegraphics[width=0.35\textwidth]{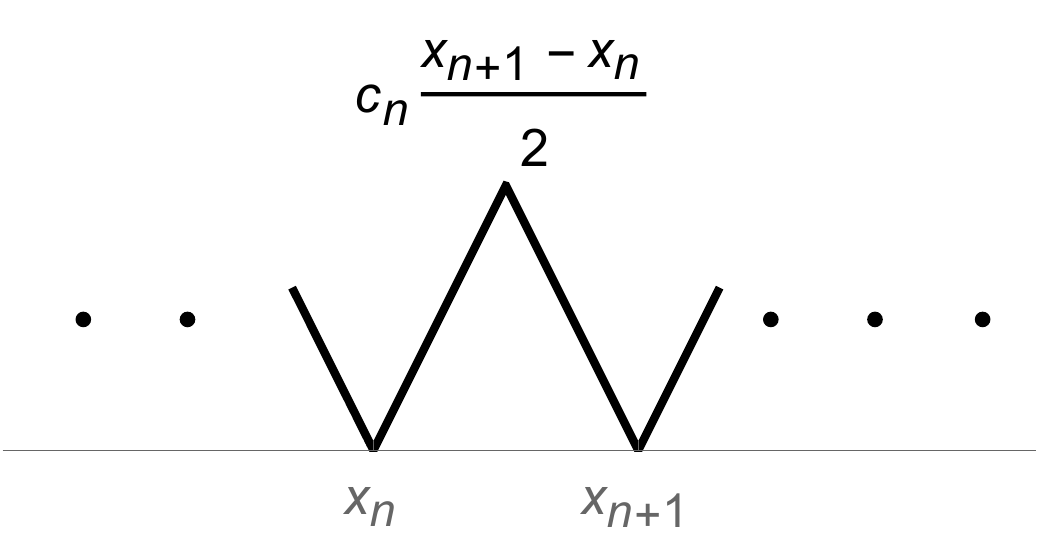}

\caption{\label{fig:stooth}The saw-tooth function.}
\end{figure}

Choose $\left\{ c_{n}\right\} _{n\in\mathbb{N}}$ such that 
\begin{equation}
\sum_{n\in\mathbb{N}}\left|c_{n}\right|^{2}\left(x_{n+1}-x_{n}\right)<\infty.\label{eq:sp11}
\end{equation}
Admissible choices for the slope-values $c_{n}$ include 
\[
c_{n}=\frac{1}{n\sqrt{x_{n+1}-x_{n}}},\;n\in\mathbb{N}.
\]

We will now show that $f\in\mathscr{H}\left(K\right)$. For the distribution
derivative computed from (\ref{eq:sp9}), we get 
\begin{equation}
\raisebox{-12mm}{\includegraphics[width=0.7\textwidth]{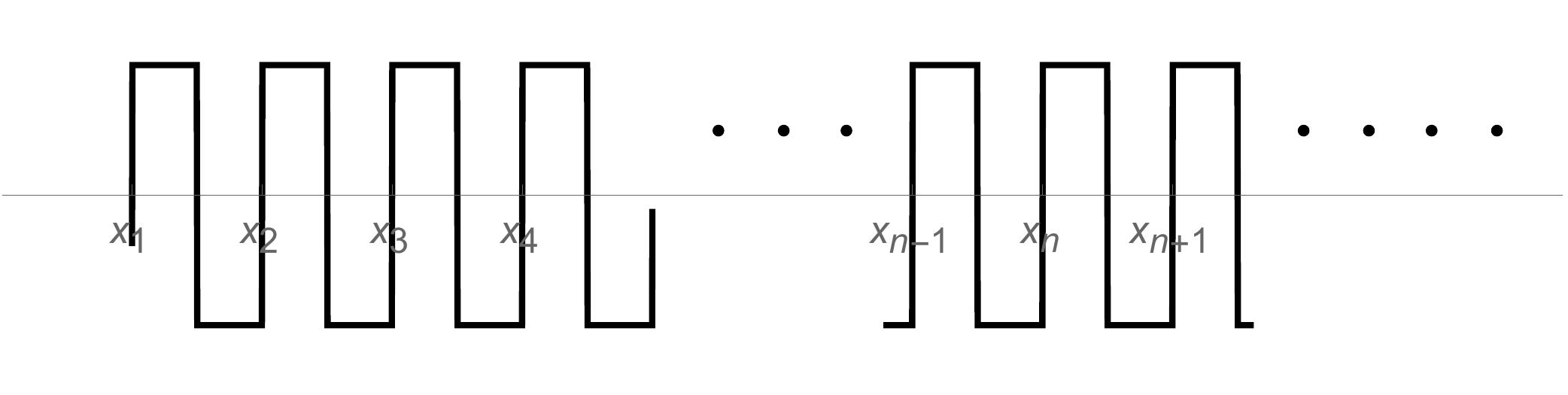}}\label{eq:sp9b}
\end{equation}
\[
\int_{0}^{\infty}\left|f'\left(x\right)\right|^{2}dx=\sum_{n\in\mathbb{N}}\left|c_{n}\right|^{2}\left(x_{n+1}-x_{n}\right)<\infty
\]
which is the desired conclusion, see (\ref{eq:sp9}).
\end{proof}
\begin{cor}
For the kernel $K\left(x,y\right)=x\wedge y$ in (\ref{eq:sp6}),
$X=[0,\infty)$, the following holds:

Given $\left\{ x_{j}\right\} _{j\in\mathbb{N}}\subset\mathbb{R}_{+}$,
$\left\{ y_{j}\right\} _{j\in\mathbb{N}}\subset\mathbb{R}$, then
the interpolation problem 
\begin{equation}
f\left(x_{j}\right)=y_{j},\;f\in\mathscr{H}\left(K\right)\label{eq:ip1}
\end{equation}
is solvable if 
\begin{equation}
\sum_{j\in\mathbb{N}}\left(y_{j+1}-y_{j}\right)^{2}/\left(x_{j+1}-x_{j}\right)<\infty.\label{eq:sp2}
\end{equation}
\end{cor}

\begin{proof}
Let $f$ be the piecewise linear spline (see \figref{ip}) for the
problem (\ref{eq:ip1}), see \figref{ip}; then the $\mathscr{H}\left(K\right)$-norm
is as follows: 
\[
\int_{0}^{\infty}\left|f'\left(x\right)\right|^{2}dx=\sum_{j\in\mathbb{N}}\left(\frac{y_{j+1}-y_{j}}{x_{j+1}-x_{j}}\right)^{2}\left(x_{j+1}-x_{j}\right)<\infty
\]
when (\ref{eq:sp2}) holds. 
\end{proof}
\begin{figure}[H]
\includegraphics[width=0.35\textwidth]{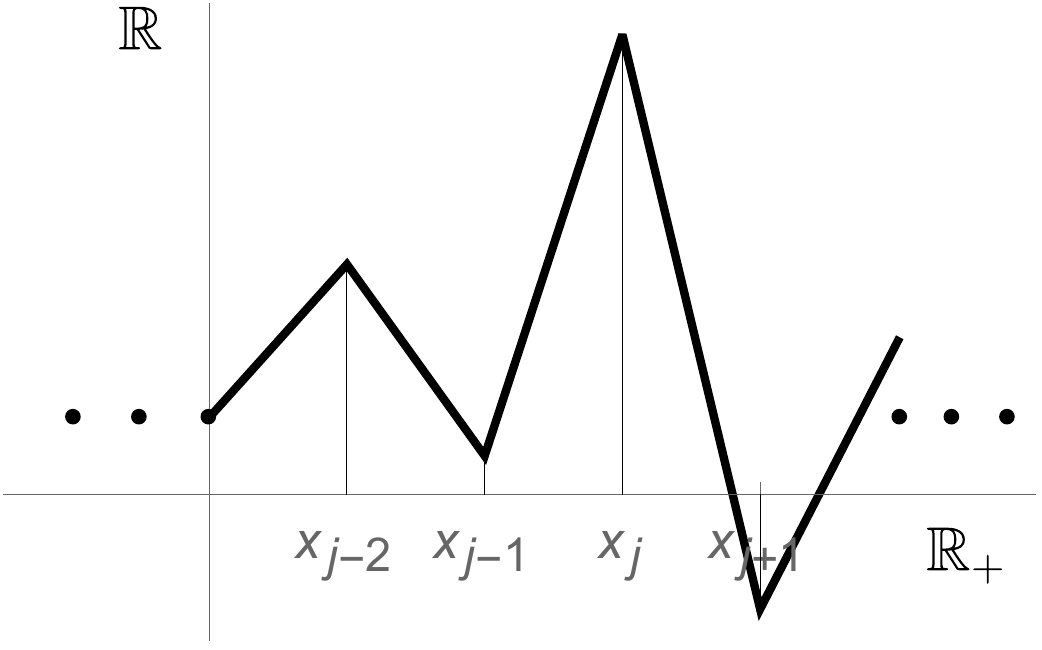}

\caption{\label{fig:ip}Piecewise linear spline.}
\end{figure}

\begin{rem}
Let $K$ be as in (\ref{eq:sp6}), $X=[0,\infty)$. For all $0\leq x_{j}<x_{j+1}<\infty$,
let 
\begin{align*}
f_{j}\left(x\right): & =\frac{2}{x_{j+1}-x_{j}}\left(K\left(x-x_{j},\frac{x_{j+1}-x_{j}}{2}\right)-K\left(x-\frac{x_{j}+x_{j+1}}{2},\frac{x_{j+1}-x_{j}}{2}\right)\right)\\
 & =\raisebox{-5mm}{\includegraphics[width=0.4\textwidth]{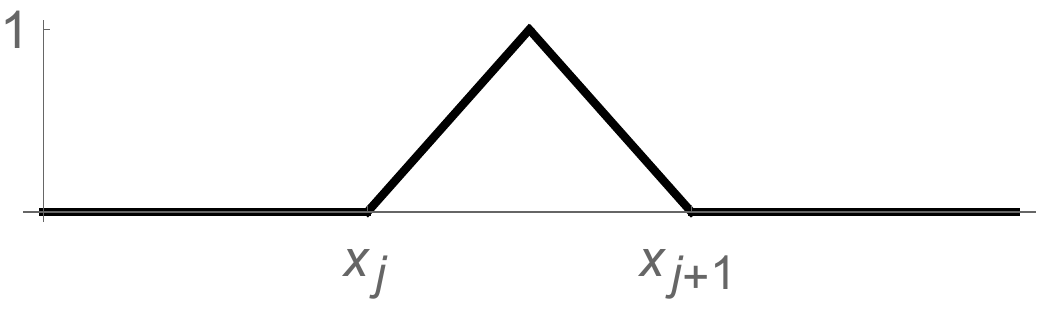}}
\end{align*}
Assuming (\ref{eq:sp11}) holds, then 
\[
f\left(x\right)=\sum_{j}c_{j}f_{j}\left(x\right)\in\mathscr{H}\left(K\right).
\]
\end{rem}

\begin{thm}
Let $X$ be a set of cardinality $c$ of the continuum, and let $K:X\times X\rightarrow\mathbb{R}$
be a positive definite kernel. Let $S=\left\{ x_{j}\right\} _{j\in\mathbb{N}}$
be a discrete subset of $X$. Suppose there are weights $\left\{ w_{j}\right\} _{j\in\mathbb{N}}$,
$w_{j}\in\mathbb{R}_{+}$, such that 
\begin{equation}
\left(f\left(x_{j}\right)\right)\in l^{2}\left(\mathbb{N},w\right)\label{eq:c1}
\end{equation}
for all $f\in\mathscr{H}\left(K\right)$. Suppose further that there
is a point $t_{0}\in X\backslash S$, a $y_{0}\in\mathbb{R}\backslash\left\{ 0\right\} $,
and $\alpha\in\mathbb{R}_{+}$ such that the infimum 
\begin{equation}
\inf_{f\in\mathscr{H}\left(K\right)}\left\{ \sum\nolimits _{j}w_{j}\left|f\left(x_{j}\right)\right|^{2}+\left|f\left(t_{0}\right)-y_{0}\right|^{2}+\alpha\left\Vert f\right\Vert _{\mathscr{H}\left(K\right)}^{2}\right\} \label{eq:c2}
\end{equation}
is strictly positive. 

Then $S$ is \uline{not} a interpolation set for $\left(K,X\right)$. 
\end{thm}

\begin{proof}
This results follows from \lemref{fr} and \thmref{ps} above. We
also refer readers to \cite{MR3670916}.
\end{proof}
\begin{acknowledgement*}
The co-authors thank the following colleagues for helpful and enlightening
discussions: Professors Daniel Alpay, Sergii Bezuglyi, Ilwoo Cho,
Myung-Sin Song, Wayne Polyzou, and members in the Math Physics seminar
at The University of Iowa.
\end{acknowledgement*}
\bibliographystyle{amsalpha}
\bibliography{ref}

\end{document}